\documentclass{article}
\usepackage{amssymb}
\usepackage{amsmath}
\usepackage{amsthm}

\usepackage{amsthm}

\newtheoremstyle{boldremark}% název stylu
  {\topsep}    % mezera nad
  {\topsep}    % mezera pod
  {\normalfont}  % font těla
  {}           % odsazení
  {\bfseries}  % font nadpisu
  {.}          % itečka po nadpisu
  {0.5em}      % mezera po nadpisu
  {}           % specifikace nadpisu

\theoremstyle{boldremark}
\usepackage{amsmath, amsthm}
\newtheorem{theorem}{Theorem}[section] 
\newtheorem{proposition}{Proposition}[section]

\newtheorem{corollary}[theorem]{Corollary}

\newtheorem{remark}[theorem]{Remark}

\usepackage[english]{babel}

\usepackage[letterpaper,top=2cm,bottom=2cm,left=3cm,right=3cm,marginparwidth=1.75cm]{geometry}

\renewcommand\[{\begin{equation}}
\renewcommand\]{\end{equation}}

\allowdisplaybreaks[4]
\usepackage{amsmath}
\usepackage{comment}
\usepackage{graphicx}
\usepackage[colorlinks=true, allcolors=blue]{hyperref}

  \theoremstyle{definition}

  \numberwithin{equation}{section}

\begin{document}
\title{Analytic continuation of weighted $H$-harmonic Bergman spaces}

\author{Matěj Moravík}

\maketitle
\begin{abstract}
We provide a partial answer to Problems 1 and 2 raised in the recent article by Blaschke et al.,
concerning the analytic continuation of weighted $H$-harmonic Bergman spaces. These are spaces
of functions annihilated by the Möbius-invariant Laplacian on the unit ball. More precisely, we identify some of the discrete Wallach sets and show, among others, that structure depends on the parity of the dimension.
\end{abstract}

\section{Introduction}
Let $B^n$ be the open unit ball in $\mathbb{R}^n$, $n > 2$. The orthogonal transformations
\begin{equation*}
    x \mapsto Ux \quad x \in \mathbb{R}^n,\; U \in O(n),
\end{equation*}
map $B^n$ and its boundary $\partial B^n$ (unit sphere) onto themselves, and so do M{\"o}bius transformations
\begin{equation*}
    \phi_a(x) := \frac{a|x-a|^2 + (1-|a|^2)(a-x)}{1 - 2\langle x, a \rangle + |x|^2|a|^2}
\end{equation*}
interchanging the origin $\textbf{0} \in B^n$ with the point $a \in B^n$; here $\langle x,y\rangle$ denotes the usual scalar product.\\
The group $G$ generated by the $\phi_a,\;a \in B^n,$ and $U \in O(n)$ is called the Möbius group of $B^n$.
Let $\mu$ denote the normalized Lebesgue measure on $B^n$, \emph{i.e.} $\mu(B^n) = 1$. Similarly, $\sigma$ will denote the surface measure normalized so that $\sigma(\partial B^n) = 1.$\\
The hyperbolic Laplacian is then defined as \begin{equation} \label{eq1}
    \Delta_hf(x) := \Delta(f \:\circ \: \phi_x)(0), \quad f\in C^2(B^n), \: x\in B^n;
\end{equation}
it follows from a direct calculation that  \begin{equation*}
    \Delta_h f(x) = (1-|x|^2)[(1- |x|^2)\Delta f(x) + 2(n-2)\langle x,\nabla f(x) \rangle ],
\end{equation*}
where $\Delta f =  \frac{\partial^2 f}{\partial x_1^2} + ...+ \frac{\partial^2 f}{\partial x_n^2}$ and $\nabla f = \big(\frac{\partial f}{\partial x_1} , ...,\frac{\partial f}{\partial x_n}\big)$.
Functions on $B^n$ annihilated by $\Delta_h$ are called hyperbolic harmonic, or $H$-harmonic for short. The vector space of all $H$-harmonic functions on the unit ball $B^n$ will be denoted by $\mathcal{H}$.
For any $s > -1$, one can consider the $\textit{weighted $H$-harmonic Bergman space}$
\begin{equation} \label{Bergman}
    \mathcal{H}_s := \{f \in L^2(B^n, d\rho_s), \quad f \text{ is $H$-harmonic} \},
\end{equation}
where $$d\rho_s(x) = \frac{\Gamma(\frac{n}{2}+s+1)}{\pi^{\frac{n}{2}}\Gamma(s+1)}(1-|x|^2)^sd\mu(x)$$
is a rotation invariant probability measure.

\begin{comment}
    It follows by a standard argument
that the point evaluations $f \mapsto f(x)$ at any $x \in B^n$ are continuous linear functionals on each $\mathcal{H}_s$, $s > -1$, making $\mathcal{H}_s$ into the \textit{weighted $H$-harmonic Bergman kernel} space, with reproducing kernel $K_s(x, y)$ defined on
$B^n \times B^n$, $H$-harmonic in both variables such that
$$
f(x) = \int_{B^n} f(y)K_s(x, y) \, d\rho_s(y) \qquad \forall x \in B^n, \forall f \in \mathcal{H}_s.
$$
$K_h^s(x,y)$ is explicitly given by
$$
K_h^s(x, y) = \sum_{m=0}^\infty \frac{S_m(|x|^2)\, S_m(|y|^2)}{I_m(s)}\, Z_m(x, y),
$$
where the function $S_m(t)$ is defined by
$$
S_m(t) := \frac{(n-1)_m}{\left(\frac{n}{2}\right)_m} \, {}_2F_1\left(\begin{matrix} m,\, 1 - \frac{n}{2} \\ m + \frac{n}{2} \end{matrix} ;\, t \right),
$$
and $I_m(s)$ is given by
$$
I_m(s) := \frac{\Gamma\left(\frac{n}{2}+s+1\right) }{\Gamma\left(\frac{n}{2}\right) \Gamma(s + 1)} \int_0^1 t^{m + \frac{n}{2} - 1} (1 - t)^s S_m(t)^2 \, dt.
$$
$$
K_h^s(x, y) = \sum_{m=0}^\infty \frac{S_m(|x|^2)\, S_m(|y|^2)}{I_m(s)}\, Z_m(x, y),
$$
Here, \(Z_m(x, y)\) denotes the zonal harmonic of degree \(m\), and $(x)_m=x(x+1)\dots(x+m-1)$ is the usual Pochhammer symbol.\\
\end{comment}
We consider an analytic continuation and the ``residue spaces'' \(\mathcal{H}_s\)
for \(s\leq-1\), in a sense that will be made precise below.

We first recall the analogous situation for weighted Bergman spaces of \textit{holomorphic}, rather than $H$-harmonic, functions on the unit ball $\mathbf{B}^n \subset \mathbb{C}^n$.
Let
$$
\mathcal{A}_s := \{ f \in L^2(\mathbf{B}^n, d\rho_s) : f \text{ is holomorphic on } \mathbf{B}^n \},
$$
where the measure $d\rho_s$ is the same as above. It is well known, that the spaces $\mathcal{A}_s$ are reproducing kernel Hilbert spaces and the reproducing kernels are given by 
\begin{equation}\label{Hol_Kernel}
K_s^{\mathrm{hol}}(x, y) = (1 - \langle x, y \rangle)^{-n-1-s}.
\end{equation}
Weighted Bergman kernels $K_s^{hol}(x,y)$, $s > -1,$ continue to be positive definite in the sense of Aronszajn \cite{Aro} for all $s \geq -n-1$, thus yielding an analytic continuation of $\mathcal{A}_s,$ for $s \in [-n-1, + \infty)$. One calls the interval $[-n-1, +\infty)$ the \textit{Wallach set} of $\mathbf{B}^n$. 
For general case of holomorphic functions defined on bounded symmetric domains, see~\cite{Rossi}.

When $s = -n-1$, the \eqref{Hol_Kernel} reduces to constant and the corresponding Hilbert space is trivial.
However, we can consider the space that arises as the ``residue'' of $\mathcal{A}_s$ at $s = -n-1$, namely, 
$$
\lim_{s \rightarrow-n-1}\frac{K^{hol}_s(x,y)-1}{s+n+1} = \ln{\frac{1}{1-\langle x, y\rangle}},
$$
which is a positive definite function on $\textbf{B}^n \times \textbf{B}^n,$ and the associated reproducing kernel Hilbert space becomes the \textit{Dirichlet space} on $B^n$; we say that $-n-1$ belongs to the \textit{discrete Wallach set}. \\

We now define the analogues of residue spaces in the case of $H$-harmonic functions on the real ball.\\
For an integer $m \geq 0$, let $\mathcal{H}^m$ be the space of restrictions to the unit sphere $\partial B^n$ of harmonic homogeneous polynomials on $\mathbb{R}^n$ of degree $m$. Then we have the following Peter-Weyl decomposition 
\begin{equation}
    L^2(\partial B^n, d\sigma) = \bigoplus_{m=0}^\infty \mathcal{H}^m
\end{equation}
under the action of the orthogonal group $O(n)$ of rotations of $\mathbb{R}^n$ \cite[Chapter~5]{Axler}.
Similarly the Bergman space of $H$-harmonic functions $\mathcal{H}_s$ admits a Peter-Weyl decomposition given by
\begin{equation}
    \mathcal{H}_s = \bigoplus_{m=0}^\infty \textbf{{H}}^m,
\end{equation}
where $\textbf{H}^m$ is the space of ``solid harmonics"
$$
H^m = \{ f \in C(\overline{B^n}) : f \text{ is } H\text{-harmonic on } B^n
\text{ and } f|_{\partial B^n} \in \mathcal{H}^m \}
$$
and the norm of $f = \sum_mf_m$, $f_m\in\textbf{H}^m$, is given by 
\begin{equation}
    \|f\|_s^2 = \sum_{m=0}^\infty I_m(s) \|f_m\|^2_{\partial B^n},
\end{equation}
where the coefficients $I_m(s)$ are equal to
\begin{equation}
I_m(s) := \frac{\Gamma\left(\frac{n}{2}+s+1\right) }{\Gamma\left(\frac{n}{2}\right) \Gamma(s + 1)} \int_0^1 t^{m + \frac{n}{2} - 1} (1 - t)^s S_m(t)^2 \, dt,
\end{equation}
with 
$$
S_m(t) := \frac{(n-1)_m}{\left(\frac{n}{2}\right)_m} \, {}_2F_1\left(\begin{matrix} m,\, 1 - \frac{n}{2} \\ m + \frac{n}{2} \end{matrix} ;\, t \right),
$$
and
$$
f_m(r\eta)=S_m(r^2)r^m f_m(\eta), 
\qquad f_m|_{\partial B^n}\in \mathcal{H}^m, 
\qquad 0\le r\le 1,\ \eta\in \partial B^n,
$$
where $_2F_1$ denotes the Gauss hypergeometric function and $(a)_m := (a)(a+1)\cdots(a+m-1)$ is the usual Pochhammer symbol. $I_m(s)=1$ corresponds to the classical Hardy space.\\
For a detailed exposition of the facts stated above, the reader is referred to \cite[Chapter~10]{Stoll2016} or \cite{Stoll2019}.

In \cite[Proposition 1]{BEY2025}  it was shown that $I_m(s)$ admits an analytic continuation for $s \le -1$, more precisely, the authors showed that $(\Gamma(n+s)\Gamma(\frac{n}{2} +1 +s) \Gamma(2n-1+s)^2)^{-1} I_m(s)$ extends to an analytic function on the entire complex plane. 

Suppose that $s_0$ is a pole of $I_m(s)$ of order $k$. Then we can consider the following space of $H$-harmonic functions on $B^n$:
\begin{equation}
    \mathcal{H}^{\#s_0}_r
    =
    \left\{
    f = \sum_{m=1}^{\infty} f_m \ \text{$H$-harmonic on $B^n$}
    :
    \|f\|_{\mathcal{H}_r^{\#s_0}}^2 < \infty
    \right\},
\end{equation}
where
\begin{equation}\label{norma}
    \|f\|_{\mathcal{H}_r^{\#s_0}}^2
    =
    \pm\sum_{m=0}^\infty c_{-r,m}(s_0)\,
    \|f_m\|_{\partial B^n}^2,
\end{equation}
Here the sign before the sum is chosen so that the resulting quantity defines a seminorm,
\(c_{-r,m}(s_0)\) denotes the \((-r)\)-th coefficient in the Laurent expansion of \(I_m(s)\) at \(s=s_0\),
and \(f_m \in \mathbf{H}^m\) is the \(m\)-th term in the Peter--Weyl decomposition of \(f\).
The expression in \eqref{norma} defines a seminorm whenever
\(c_{-r,m}(s_0) \geq 0\) for all \(m \geq 0\), or
\(c_{-r,m}(s_0) \leq 0\) for all \(m > 0\).
This leads to the following definition of the discrete Wallach sets.
\begin{equation*}
    \mathcal{W}^r_d := \{s \in \mathbb{C}\,:\, \text{$I_m(s)$ has pole at $s$, and $c_{-r,m}(s)$ are all nonnegative or all nonpositive, $\forall m \ge 0$} \}.
\end{equation*}
For example it was shown in \cite[Proposition~1]{BEY2025} that $-n,\, -n-1 \in \mathcal{W}_1^d,$ furthermore $\mathcal{H}^{\#-n}_1$ and $\mathcal{H}^{\#-n-1}_1$ correspond to the spaces with unique Möbius invariant semi-inner product.\\
For an analogous situation in the case of $M$-harmonic functions see \cite{EYM}.

In \cite{Ureyen2023} the following estimate for $I_m(s)$, was established:
\begin{equation}\label{Urye}
    c_s < m^{s+1} I_m(s) < C_s \quad \forall m \in \mathbb{N} \quad \text{and $s \geq -1$},
\end{equation}
for some absolute constants $C_s > c_s > 0.$ 
However, somewhat surprisingly, it was shown in \cite{BEY2025} that
$I_m(-2)=1$ for all $m\geq0$. Hence, in the case $s=-2$,
the space $\mathcal{H}_s$ reduces to the Hardy space, and the estimate
\eqref{Urye} fails. This is a situation that has no parallel in the case of holomorphic or $M$-harmonic functions. 
This observation led the authors of \cite{BEY2025} to introduce the spaces
\begin{equation}\label{definiceH}
\mathcal{H}_{\#s}
:=
\left\{
f = \sum_{m=0}^{\infty} f_m
\,:\,
f \text{ is } H\text{-harmonic on } B^n
\text{ and }
\|f\|_{\#s}^{2} < \infty
\right\},
\end{equation}
where
$$
\|f\|_{\#s}^{2}
:=
\sum_{m=0}^{\infty}
(m+1)^{-s-1}
\|f_m\|_{\partial B^n}^{2}.
$$
Here again \(f=\sum_{m=0}^{\infty} f_m\) denotes the Peter--Weyl decomposition of \(f\).
They then asked for which $s<-1$ the spaces $\mathcal{H}_{\#s}$ and
$\mathcal{H}_s$ coincide, or, equivalently, for which $s < -1$ the estimate \eqref{Urye} holds (with $I_m(s)$ possibly replaced by $-I_m(s)).$

In this article, we derive an analytic continuation of $I_m(s)$ in terms
of the generalized hypergeometric series ${}_4F_3$.
This allows us to give a partial answer to the questions raised in \cite{BEY2025} concerning the
characterisation of $\mathcal{W}_k^d$. 
We show that $\mathcal{W}_k^d$
depends on the parity of $n$, which is again a situation with no
counterpart in the holomorphic or $M$-harmonic case.
We show that, for even $n \geq 6$, one has
$-n+1,-n+2 \in \mathcal{W}_1^d$, and that
$\mathcal{W}_r^d=\emptyset$ for $r>1$. For $n=4$, only $-3 \in \mathcal{W}_1^1.$\\
For odd $n$, we show that
$-n+\frac{1}{2},-n-\frac{1}{2}\in \mathcal{W}_1^d$,
$$
\mathcal{W}_2^d=\{-2n+1-k:\ k = 0,1,2,\dots\},
$$
and $\mathcal{W}_r^d=\emptyset$ for $r>2$.\\
In the same article, the authors asked about the behaviour of $I_m(s)$
for $s<-1$ as $m\to\infty$. We prove that, for odd $n$,
$m^{s+1}I_m(s)$ converges to a nonzero limit as $m\to\infty$ whenever
$s+n-1\geq0$, $s\notin\{-2,-3,\ldots\}$, and
$-\frac n2-s\neq 0,1,2,\dots$. We also show that the condition
$s+n-1\geq0$ can be omitted when $n$ is even.
\section{Notation and preliminaries}
The notation ${}_p F_q$ denotes the generalized hypergeometric series, defined as
$$
{}_p F_q\left( \begin{matrix} a_1, a_2, \dots, a_p \\[3pt] b_1, b_2, \dots, b_q \end{matrix} \;;\; z \right) = \sum_{\ell=0}^{\infty} \frac{(a_1)_\ell (a_2)_\ell \cdots (a_p)_\ell}{(b_1)_\ell (b_2)_\ell \cdots (b_q)_\ell} \frac{z^\ell}{\ell!},
$$
provided that no $b_j$ is a non-positive integer,
and $(a)_\ell = \frac{\Gamma(a+\ell)}{\Gamma(a)} = a(a+1)\cdots(a+\ell-1)$ denotes the usual Pochhammer symbol (rising factorial).\\
The $_{p+1}F_{p}$ is absolutely convergent for $z=1$, provided that $\sum_{j=1}^p b_j - \sum_{i=1}^{p+1} a_i > 0$. 

We also encounter the following hypergeometric series:
$$
{}_p F_q\left( \begin{matrix} -m, a_2, \dots, a_p \\[3pt] -n, b_2, \dots, b_q \end{matrix} \;;\; z \right).
$$
For positive integers $m$ and $n$ with $m < n$, this will always denote, without explicitly mentioning it, the finite sum from $\ell = 0$ to $m$.

The hypergeometric series ${}_{p+1}F_p$ is called \textit{balanced} (\textit{Saalschützian}) when $\sum_{j=1}^p b_j - \sum_{i=1}^{p+1} a_i = 1$.
We also make use of the following identity:
\begin{align}\label{TheIdentity}
   &\lim_{\epsilon \to 0} \epsilon \, {}_p F_q\left( \begin{matrix} a_1, \dots, a_p \\[3pt] b_1, \dots, b_{q-1}, -m+\epsilon \end{matrix} \;;\; z \right) \nonumber \\
   &\quad = \frac{(-1)^m z^{m+1}}{m! (m+1)!} \frac{\prod_{i=1}^p (a_i)_{m+1}}{\prod_{j=1}^{q-1} (b_j)_{m+1}} \, {}_p F_q\left( \begin{matrix} a_1+m+1, \dots, a_p+m+1 \\ b_1+m+1, \dots, b_{q-1}+m+1, m+2 \end{matrix} \;;\; z \right),
\end{align}
provided that no $a_i$ is a negative integer greater than $-m$.
This can be deduced from the fact that $\lim_{\epsilon \to 0}\frac{\epsilon}{(-m+\epsilon)_\ell}$ equals $0$ for $0\leq \ell \leq m$, and $\frac{(-1)^m}{m!\,(\ell-m-1)!}$ for $\ell > m$.

We also need a two-variable generalization of the hypergeometric series, 
namely the \textit{Kampé de Fériet} function:
\[
F^{p:q;r}_{s:t;u} \left(
\begin{matrix}
a_1,\dots,a_p : b_1,\dots,b_q ; b'_1,\dots,b'_r \\[3pt]
c_1,\dots,c_s : d_1,\dots,d_t ; d'_1,\dots,d'_u
\end{matrix}
; x,y
\right)
=
\sum_{m=0}^{\infty}\sum_{\ell=0}^{\infty}
\frac{\prod_{i=1}^{p}(a_i)_{m+\ell}}
     {\prod_{i=1}^{s}(c_i)_{m+\ell}}
\frac{\prod_{i=1}^{q}(b_i)_m}
     {\prod_{i=1}^{t}(d_i)_m}
\frac{\prod_{i=1}^{r}(b'_i)_\ell}
     {\prod_{i=1}^{u}(d'_i)_\ell}
\frac{x^m y^\ell}{m!\ell!}.
\]
In particular, we will need the special case
$$
F^{0:3;3}_{1:1;1} \left(
\begin{matrix}
- : a\,,b\,,c\, ;\,a'\,,b'\,,c' \\[3pt]
d\, : e\,; e'
\end{matrix}
; x,y
\right).
$$
This series is absolutely convergent at \(x=y=1\), provided that
$$
d+e-a-b-c>0
\qquad\text{and}\qquad
d+e'-a'-b'-c'>0,
$$
see~\cite{pitre1996}.

We shall also use the big-\(\mathcal{O}\) notation. We write
$$
a_m = \mathcal{O}(b_m)
$$
if there exists a constant \(C>0\) such that
$$
|a_m| \leq C |b_m|
$$
for all sufficiently large \(m\).

In order to make the notation more compact, we shall sometimes use the abbreviation
$$
\frac{\Gamma(a_1)\cdots\Gamma(a_n)}
     {\Gamma(b_1)\cdots\Gamma(b_m)}
=
\Gamma\!\left(
\begin{matrix}
a_1,\dots,a_n
\\
b_1,\dots,b_m
\end{matrix}
\right).
$$
\section{Analytic continuation of $I_m(s)$}

  \begin{proposition}\label{Pro 3.1}
      When $n$ is even we have:
      \begin{align} \label{sude}
I_m(s) &= (n-1)_m^2  \Gamma\!\left(
\begin{matrix}
\frac{n}{2},\; s+2n-1,\; m+\frac{n}{2}-1,\; s+n
\\
m,\; m+n-1,\; s+\frac{3n}{2},\; m+s+\frac{3n}{2}-1
\end{matrix}
\right) \nonumber \\[3. pt]
&\quad \times {}_4F_3 \left( \begin{matrix} \frac{n}{2}, \, s - m + \frac{n}{2} + 1, \, 1 - \frac{n}{2}, \, s + n \\ s + \frac{3n}{2}, \, 2 - m - \frac{n}{2}, \, s + \frac{n}{2} + 1 \end{matrix} ; 1 \right),
      \end{align}
      for $m >0.$\\
      When $n$ is odd we have 
\begin{align} 
I_m(s) &= \frac{(n-1)_m^2}{(n+s)_m}\,
\Gamma\!\left(
\begin{matrix}
s+\frac{n}{2}+1,\; s+2n-1,\; 1-m-\frac{n}{2}
\\
s-m+\frac{n}{2}+1,\; 1-\frac{n}{2},\; m+2n+s-1
\end{matrix}
\right) \label{liche} \\[3. pt]
&\qquad \quad \times {}_4F_3 \left( \begin{matrix} n + s, \, m + n - 1, \, m + s + \frac{3n}{2} - 1, \, m \\ m + n + s, \, m + \frac{n}{2}, \, m + 2n + s - 1 \end{matrix} ; 1 \right) \nonumber \\[1em]
& + (n-1)_m^2  \Gamma\!\left(
\begin{matrix}
\frac{n}{2},\; s+2n-1,\; m+\frac{n}{2}-1,\; s+n
\\
m,\; m+n-1,\; s+\frac{3n}{2},\; m+s+\frac{3n}{2}-1
\end{matrix}
\right) \nonumber \\[3. pt]
&\qquad \quad \times {}_4F_3 \left( \begin{matrix} \frac{n}{2}, \, s - m + \frac{n}{2} + 1, \, 1 - \frac{n}{2}, \, s + n \\ s + \frac{3n}{2}, \, 2 - m - \frac{n}{2}, \, s + \frac{n}{2} + 1 \end{matrix} ; 1 \right) \nonumber,
\end{align}
for $m >0.$
  \end{proposition}
  \begin{proof}
Recall that
\begin{equation*}
I_m(s) := \frac{\Gamma(\frac{n}{2} + s + 1)}{\Gamma(\frac{n}{2})\Gamma(s + 1)} \int_{0}^{1} t^{m + \frac{n}{2} - 1} (1 - t)^{s} \, S_m(t)^{2} \, dt,
\end{equation*}
where
\begin{equation*}
S_m(t) := \frac{(n-1)_m}{(\frac{n}{2})_m} \sum_{k=0}^{\infty} \frac{(m)_k (1 - \frac{n}{2})_k}{(m + \frac{n}{2})_k k!} t^k.
\end{equation*}
For $m=0$, we have $I_m(s)=1$ for all $s\in\mathbb{C}$.
Hence, from now on, we assume that $m>0$; moreover, for the absolute convergence of various sums, we may assume that $s > -1$.\\
Expanding each $S_m(t)$ into sum, changing order of summation and integration and using
$$
\int_{0}^{1} t^{\beta}(1-t)^{\alpha}\,dt
= \frac{\Gamma(\beta+1)\,\Gamma(\alpha+1)}{\Gamma(\alpha+\beta+2)},
\quad \text{for $\alpha>-1, \; \beta>-1$},
$$
to integrate term by term we obtain:
\begin{align}\notag
I_m(s)
&=\frac{\Gamma\!\left(\frac n2+s+1\right)}{\Gamma\!\left(\frac n2\right)\Gamma(s+1)}
\left(\frac{(n-1)_m}{\left(\frac n2\right)_m}\right)^2
\sum_{k=0}^{\infty}\sum_{j=0}^{\infty}
\frac{(m)_k\left(1-\frac n2\right)_k}{\left(m+\frac n2\right)_k\,k!}
\frac{(m)_j\left(1-\frac n2\right)_j}{\left(m+\frac n2\right)_j\,j!}\notag \\
&\quad\times
\int_0^1
t^{m+\frac n2-1+k+j}(1-t)^s\,dt\notag
\\
&=\frac{\Gamma\!\left(\frac n2+s+1\right)}{\Gamma\!\left(\frac n2\right)\Gamma(s+1)}
\left(\frac{(n-1)_m}{\left(\frac n2\right)_m}\right)^2
\sum_{k=0}^{\infty}\sum_{j=0}^{\infty}
\frac{(m)_k\left(1-\frac n2\right)_k}{\left(m+\frac n2\right)_k\,k!}
\frac{(m)_j\left(1-\frac n2\right)_j}{\left(m+\frac n2\right)_j\,j!}\label{sumapom} \\
&\quad\times
\frac{\Gamma\!\left(m+\frac n2+k+j\right)\Gamma(s+1)}
{\Gamma\!\left(m+\frac n2+k+j+s+1\right)}\notag.
\end{align}
After cancelling the factor $\Gamma(s+1)$ and rewriting
$\Gamma\left(m+\frac n2+k+j\right)$ and $\Gamma\left(m+\frac n2+k+j+s+1\right)$ as
$\Gamma\left(m+\frac n2+j\right)\left(m+\frac n2+j\right)_k$ and
$\Gamma\left(m+\frac n2+s+1+j\right)\left(m+\frac n2+s+1+j\right)_k$,
the sum over $k$ in \eqref{sumapom} can be expressed as a hypergeometric series ${}_3F_2$:\\
\begin{align}\notag
I_m(s) 
&= \frac{\Gamma(\frac{n}{2} + s + 1)}{\Gamma(\frac{n}{2})} \left( \frac{(n-1)_m}{(\frac{n}{2})_m} \right)^2 \sum_{j=0}^{\infty} \frac{(m)_j (1 - \frac{n}{2})_j}{(m + \frac{n}{2})_j j!} \left( \frac{\Gamma(m + \frac{n}{2} + j)}{\Gamma(m + \frac{n}{2} + s + 1 + j)} \right)  \\\label{suma3F2}
&\times {}_3F_2\left(
\begin{matrix}
m,\; 1 - \frac{n}{2},\; m+\frac{n}{2}+j\\
m + \frac{n}{2},\; m+\frac{n}{2}+s+1+j
\end{matrix}
;\; 1
\right).
\end{align}
Now, by applying Thomae's identity (\cite[p. 14, eq. 2]{bailey1935}):
\begin{equation}\label{Thomae}
\frac{\Gamma(a)}{\Gamma(e)\Gamma(f)}\,
{}_3F_2\!\left(\begin{matrix}
a,\; b,\; c\\
e,\; f
\end{matrix};\,1\right)
=
\frac{\Gamma(\sigma)}{\Gamma(\sigma+b)\Gamma(\sigma+c)}\,
{}_3F_2\!\left(\begin{matrix}
\sigma,\; e-a,\; f-a\\
\sigma+b,\; \sigma+c
\end{matrix};\,1\right),
\end{equation}
with parameters
\begin{align*}
    a &= m, \quad b = 1 - \frac{n}{2}, \quad c = m + \frac{n}{2} + j, \\
    e &= m + \frac{n}{2}, \quad f = m + \frac{n}{2} + s + j + 1,
\end{align*}
where $\sigma$ denotes the \textit{excess parameter} $\sigma := e+f-a-b-c = n+s$, we see that
\begin{align}\label{3F2}
&{}_3F_2\!\left(
\begin{matrix}
m,\; 1-\frac n2,\; m+\frac n2+j\\[2pt]
m+\frac n2,\; m+\frac n2+s+1+j
\end{matrix};1\right)
=
\frac{
\Gamma\!\left(m+\frac n2\right)\,
\Gamma\!\left(m+\frac n2+s+1+j\right)\,
\Gamma(n+s)
}{
\Gamma\!\left(m+\frac n2+j\right)\,
\Gamma(m+n+s)\,
\Gamma\!\left(s+\frac n2+1\right)
}\\\notag
&
\;\times
{}_3F_2\!\left(
\begin{matrix}
n+s,\; -j,\; s+1\\[2pt]
m+n+s,\; s+\frac n2+1
\end{matrix};1\right).
\end{align}
We then apply Sheppard’s transformation:
\begin{equation*}\label{Transform}
{}_3F_2 \!\left(
\begin{matrix}
-\ell,\ \alpha_1,\ \alpha_2 \\
\beta_1,\ \beta_2
\end{matrix}
; 1 \right)
=
\frac{(\beta_2 - \alpha_1)_\ell}{(\beta_2)_\ell}
\,
{}_3F_2 \!\left(
\begin{matrix}
-\ell,\ \alpha_1,\ \beta_1 - \alpha_2 \\
\beta_1,\ 1 - \beta_2 + \alpha_1 - \ell
\end{matrix}
; 1 \right),
\end{equation*}
for terminating $_3F_2$ (see \cite[Appendix, formula (I)]{andrews1999}).
Thus, \eqref{3F2} is equal to
\begin{align}\label{3F2konec}
&\frac{
\Gamma\!\left(m+\frac n2\right)\,
\Gamma\!\left(m+\frac n2+s+1+j\right)\,
\Gamma(n+s)
}{
\Gamma\!\left(m+\frac n2+j\right)\,
\Gamma(m+n+s)\,
\Gamma\!\left(s+\frac n2+1\right)
}
\,\frac{\left(1-\frac n2\right)_j}{\left(s+\frac n2+1\right)_j} \\\notag
&\qquad\times
{}_3F_2\!\left(
\begin{matrix}
-j,\; n+s,\; m+n-1 \\[2pt]
m+n+s,\; \frac n2-j
\end{matrix}
;1\right).
\end{align}
Substituting \eqref{3F2konec} into \eqref{suma3F2} and cancelling Gamma functions we obtain:
\begin{align}\notag
I_m(s)
&= \frac{1}{\Gamma\!\left(\tfrac{n}{2}\right)}
\left( \frac{(n-1)_m}{\left(\tfrac{n}{2}\right)_m} \right)^2
\frac{\Gamma\!\left(m + \tfrac{n}{2}\right)\Gamma(n+s)}
{\Gamma(m+n+s)}
\\\label{sumapomm}
&\quad \times
\sum_{j=0}^{\infty}
\frac{(m)_j\left(1-\tfrac{n}{2}\right)_j}
{\left(m+\tfrac{n}{2}\right)_j\, j!}
\frac{\left(1-\tfrac{n}{2}\right)_j}
{\left(s+\tfrac{n}{2}+1\right)_j}
\;
{}_3F_2\!\left(
\begin{matrix}
-j,\ n+s,\ m+n-1 \\[2pt]
m+n+s,\ \tfrac{n}{2}-j
\end{matrix}
; 1 \right).
\end{align}
Expanding $_3F_2$ and rearranging the terms, we obtain
\begin{align*}
I_m(s)&=\frac{1}{\Gamma\!\left(\tfrac{n}{2}\right)}
\left( \frac{(n-1)_m}{\left(\tfrac{n}{2}\right)_m} \right)^2
\frac{\Gamma\!\left(m + \tfrac{n}{2}\right)\Gamma(n+s)}
{\Gamma(m+n+s)}\\
&\times\sum_{j=0}^{\infty}
\sum_{r=0}^{j}
\frac{(m)_j\left(1-\tfrac{n}{2}\right)_j^2}{\left(m+\tfrac{n}{2}\right)_j
\left(s+\tfrac{n}{2}+1\right)_j\, j!}
\,
\frac{(-j)_r\,(n+s)_r\,(m+n-1)_r}
{(m+n+s)_r\,\left(\tfrac{n}{2}-j\right)_r\, r!}.
\end{align*}
Making the substitution $j = k+r$, changing the order of summation, and using the identities:
\begin{equation*}
(-j)_k = (-1)^k\frac{j!}{(j-k)!},
\qquad
\left(\tfrac{n}{2}-j\right)_k
=  (-1)^k (1-\tfrac{n}{2}+r)_k=\frac{(-1)^k\left(1-\tfrac{n}{2}\right)_{r+k}}
{\left(1-\tfrac{n}{2}\right)_r},
\end{equation*}
we obtain 
\begin{align}\notag
I_m(s)&=\frac{1}{\Gamma\!\left(\tfrac{n}{2}\right)}
\left( \frac{(n-1)_m}{\left(\tfrac{n}{2}\right)_m} \right)^2
\frac{\Gamma\!\left(m + \tfrac{n}{2}\right)\Gamma(n+s)}
{\Gamma(m+n+s)}\\\label{Poslední3F2}
&\quad\times\sum_{k=0}^{\infty}\sum_{r=0}^{\infty}
\frac{(m)_{k+r}\,\left(1-\tfrac n2\right)_{k+r}\,\left(1-\tfrac n2\right)_r}
{\left(m+\tfrac n2\right)_{k+r}\,\left(s+\tfrac n2+1\right)_{k+r}}
\;
\frac{(n+s)_k\,(m+n-1)_k}{(m+n+s)_k\,k!}
\;
\frac{1}{r!}\\\notag
&=
\frac{1}{\Gamma\!\left(\tfrac{n}{2}\right)}
\left( \frac{(n-1)_m}{\left(\tfrac{n}{2}\right)_m} \right)^2
\frac{\Gamma\!\left(m + \tfrac{n}{2}\right)\Gamma(n+s)}
{\Gamma(m+n+s)}
\sum_{k=0}^{\infty}
\frac{(n+s)_k\,(m+n-1)_k}{(m+n+s)_k\,k!}
\frac{(m)_k\left(1-\tfrac n2\right)_k}{\left(m+\tfrac n2\right)_k\left(s+\tfrac n2+1\right)_k}\\\notag
&\quad\times{}_3F_2\!\left(
\begin{matrix}
m+k,\ 1-\tfrac n2+k,\ 1-\tfrac n2\\
m+\tfrac n2+k,\ s+\tfrac n2+1+k
\end{matrix}
;1\right).
\end{align}
Using the estimate
\begin{equation} \label{odhad}
\frac{\Gamma(\alpha + m)}{\Gamma(\beta + m)} = m^{\alpha - \beta}\left(1+\mathcal{O}\left(\frac{1}{m}\right)\right)
\end{equation}
(see~\cite{erdelyi1}, Eq.~(4), p.~47), we see that \eqref{Poslední3F2} is bounded, up to a multiplicative constant, by
\begin{align*}
\sum_{k,r=1}^{\infty}
(k+r)^{-s-\frac{3n}{2}} r^{-\frac n2} k^{n-2}
&=
\sum_{\substack{k,r\ge 1\\ r\le k}}
(k+r)^{-s-\frac{3n}{2}} r^{-\frac n2} k^{n-2}
+
\sum_{\substack{k,r\ge 1\\ k<r}}
(k+r)^{-s-\frac{3n}{2}} r^{-\frac n2} k^{n-2}\\
&<
\sum_{\substack{k,r\ge 1\\ r\le k}}
k^{-s-\frac{3n}{2}} r^{-\frac n2} k^{n-2}
+
\sum_{\substack{k,r\ge 1\\ k<r}}
r^{-s-\frac{3n}{2}} r^{-\frac n2} k^{n-2}\\
&=
\sum_{k=1}^{\infty}
k^{-s-\frac n2-2}
\sum_{r=1}^{k} r^{-\frac n2}
+
\sum_{r=1}^{\infty}
r^{-s-2n}
\sum_{k=1}^{r-1} k^{n-2}
\nonumber \\
&=
C\left(\sum_{k=1}^{\infty}
k^{-s-\frac n2-2}
+
\sum_{r=1}^{\infty}
r^{-s-n-1}\right),
\end{align*}
which converges absolutely for $s > -1 - \frac{n}{2}$. Hence, changing the order of summation in \eqref{Poslední3F2} is justified.\\
Applying one more Thomae's transform \eqref{Thomae} with parameters:
\begin{align*}
    & a=m+k,\;\; \quad b=k-\tfrac n2,\;\; \quad c=1-\tfrac n2,\\
    & e=m+\tfrac n2+k,\;\; \quad f=s+\tfrac n2+k+1,
\end{align*}
we obtain 
\begin{align*}
{}_3F_2\!\left(
\begin{matrix}
m+k+1,\ k+1-\tfrac n2,\ 1-\tfrac n2\\[4 pt]
m+\tfrac n2+k+1,\ s+\tfrac n2+k+2
\end{matrix}
;\,1\right)
&=
\frac{
\Gamma\!\left(m+\tfrac n2+k+1\right)\,
\Gamma\!\left(s+\tfrac n2+k+2\right)\,
\Gamma(s+2n-1)
}{
\Gamma(m+k+1)\,
\Gamma\!\left(s+\tfrac{3n}{2}+k\right)\,
\Gamma\!\left(s+\tfrac{3n}{2}\right)
}
\\
&\qquad\times
{}_3F_2\!\left(
\begin{matrix}
\tfrac n2,\ s-m+\tfrac n2+1,\ s+2n-1\\[4 pt]
s+\tfrac{3n}{2}+k,\ s+\tfrac{3n}{2}
\end{matrix}
;\,1\right).
\end{align*}
After substituting this into the sum and simplifying the Pochhammer symbols, we obtain:
\begin{align*}
I_m(s)
&= \frac{\Gamma(\frac{n}{2}) (n-1)_m^2 \Gamma(n+s) \Gamma(s+\frac{n}{2}+1) \Gamma(s+2n-1)}
{\Gamma(m) \Gamma(m+n+s) \Gamma(s+\frac{3n}{2})^2}
\notag \\[6pt]
&\quad \times
\sum_{k=0}^{\infty}
\sum_{j=0}^{\infty}
\frac{(n+s)_k (m+n-1)_k \left(1-\frac{n}{2}\right)_k}{(m+n+s)_k\,k!}
\;
\frac{
\left(\frac{n}{2}\right)_j
\left(s-m+\frac{n}{2}+1\right)_j
\left(s+2n-1\right)_j
}{
\left(s+\frac{3n}{2}\right)_{k+j}
\left(s+\frac{3n}{2}\right)_j\; j!
}.
\end{align*}
The double sum can be identified as a special case of the Kampé de Fériet function:
\begin{align}
I_m(s)
&= \frac{\Gamma(\frac{n}{2}) (n-1)_m^2 \Gamma(n+s)
\Gamma(s+\frac{n}{2}+1) \Gamma(s+2n-1)}
{\Gamma(m) \Gamma(m+n+s) \Gamma(s+\frac{3n}{2})^2}
\label{Kampee}\\
&\quad \times
F^{0:3;3}_{1:1;1}\!\left(
\begin{matrix}
-:n+s,\;m+n-1,\;1-\frac{n}{2};
\,\frac{n}{2},\;s-m+\frac{n}{2}+1,\;s+2n-1\\[3pt]
s+\frac{3n}{2}: m+n+s\, ; \,s+\frac{3n}{2}
\end{matrix}
; 1,1\right).
\notag
\end{align}
Note that \eqref{Kampee} implies that
$$
\bigl(\Gamma(n+s)\Gamma(s+\tfrac n2+1)\Gamma(s+2n-1)\bigr)^{-1} I_m(s)
$$
is an analytic function on the entire complex plane.\\
Since \(F^{0:3;3}_{1:1;1}\) depends continuously on the parameter \(m\), we may introduce a small perturbation \(m+\epsilon\). We can then apply the following reduction formula (see~\cite{pitre1996}, Eq.~(6)):
\begin{align}\label{redukce}
    & F^{0:3;3}_{1:1;1} \! \left(
  \begin{array}{c}
    - : a,\, b,\, c \, ; \, d - a,\, d-b \, c' \\
    d\, : e \, ; \, e'
  \end{array}
  ; 1, 1 \right)\\\nonumber
    & = \Gamma\!\left(
\begin{matrix}
d,\; e',\; a'+b'-d,\; d+e'-a'-b'-c'
\\[3pt]
a',\; b',\; e'-c',\; d+e'-a'-b'
\end{matrix}
\right) \\\notag
&\quad \times {}_4F_3 \left( \begin{matrix} a, b, e-c, d+e'-a'-b'-c' \\[3pt] e, 1+d-a'-b', d+e'-a'-b' \end{matrix} ; 1 \right) \\\nonumber
&+ \Gamma\!\left(
\begin{matrix}
d,\; e,\; a+b-d,\; d+e-a-b-c
\\[3pt]
a,\; b,\; e-c,\; d+e-a-b
\end{matrix}
\right) \\\nonumber
&\quad \times {}_4F_3 \left( \begin{matrix} a', b', e'-c', d+e-a-b-c \\[3pt] e', 1+d-a-b, d+e-a-b \end{matrix} ; 1 \right),
\end{align}
with $a = n + s$ and $b = (m +\epsilon) +n-1$ and we obtain
\begin{align*}
F^{0:3;3}_{1:1;1} \! &\left(
\begin{array}{c}
  - : n+s,\, m+\epsilon+n-1,\, 1-\frac{n}{2} ; \, \frac{n}{2},\, s-m-\epsilon+\frac{n}{2}+1,\, s+2n-1 \\[3pt]
  s+\frac{3n}{2}\, : m+\epsilon+n+s \, ; \, s+\frac{3n}{2}
\end{array}
; 1, 1 \right) \\
&= \frac{\Gamma(s+\frac{3n}{2})^2 \Gamma(1-m-\epsilon-\frac{n}{2}) \Gamma(m+\epsilon)}{\Gamma(\frac{n}{2}) \Gamma(s-m-\epsilon+\frac{n}{2}+1) \Gamma(1-\frac{n}{2}) \Gamma(s+m+\epsilon+2n-1)} \\
&\quad \times {}_4F_3 \left( \begin{matrix} n+s, \, m+\epsilon+n-1, \, m+\epsilon+s+\frac{3n}{2}-1, \, m+\epsilon \\[3pt] m+\epsilon+n+s, \, m+\epsilon+\frac{n}{2}, \, s+m+\epsilon+2n-1 \end{matrix} ; 1 \right) \\
&+ \frac{\Gamma(s+\frac{3n}{2}) \Gamma(m+\epsilon+n+s) \Gamma(m+\epsilon+\frac{n}{2}-1) \Gamma(s+n)}{\Gamma(n+s) \Gamma(m+\epsilon+n-1) \Gamma(m+\epsilon+s+\frac{3n}{2}-1) \Gamma(s+\frac{n}{2}+1)} \\
&\quad \times {}_4F_3 \left( \begin{matrix} \frac{n}{2}, \, s-m-\epsilon+\frac{n}{2}+1, \, 1-\frac{n}{2}, \, s+n \\[3pt] s+\frac{3n}{2}, \, 2-m-\epsilon-\frac{n}{2}, \, s+\frac{n}{2}+1 \end{matrix} ; 1 \right).
\end{align*}
When \(n\) is even, the first term vanishes due to the factor
$\frac{1}{\Gamma(1-\frac{n}{2})}.$ Therefore, letting \(\epsilon \to 0\)
and multiplying by the prefactor
\begin{equation*}
    \frac{\Gamma(\frac{n}{2}) (n-1)_m^2 \Gamma(n+s) \Gamma(s+\frac{n}{2}+1) \Gamma(s+2n-1)}
{\Gamma(m) \Gamma(m+n+s) \Gamma(s+\frac{3n}{2})^2},
\end{equation*}
Our conclusion follows for \(s>-1\). For \(s\leq -1\), the result follows by analytic continuation except at the possible poles of \(I_m(s)\), which occur when
$
s+\frac n2-1=1,2,3,\dots
$
for even \(n\), and when
$
-s-\frac n2+1=1,2,3,\dots$ or $
-s-n=1,2,3,\dots
$
for odd \(n\).
\end{proof}
The following proposition gives another formula for $I_m(s)$ for even $n$.
\begin{proposition} \label{Prop 3.2}
    When $n$ is even we have
    \begin{equation}\label{sude_dva_dva}
        I_m(s) = \frac{(n-1)_m}{(n+s)_m} \; {}_4F_3 \left( \begin{matrix} m,\; 1 - \frac{n}{2},\; n+s,\; s+1 \\[4pt] m + n + s,\; \frac{n}{2} + s + 1,\; 2-n \end{matrix} \;;\; 1 \right).
    \end{equation}
\end{proposition}
\begin{proof}
From the equation~\eqref{sumapom} we know that for $s > -1$
\begin{align}\label{2.2-1}
    I_m(s) &= 
      \frac{\Gamma\left(\frac{n}{2} + s + 1\right)}{\Gamma\left(\frac{n}{2}\right)} \left( \frac{(n-1)_m}{\left(\frac{n}{2}\right)_m} \right)^2 \sum_{j=0}^{\infty} \frac{(m)_j \left(1 - \frac{n}{2}\right)_j}{\left(m + \frac{n}{2}\right)_j j!} \left( \frac{\Gamma\left(m + \frac{n}{2} + j\right)}{\Gamma\left(m + \frac{n}{2} + s + 1 + j\right)} \right)\nonumber \\
    &\quad \times {}_3F_2 \left( \begin{matrix} m,\; 1 - \frac{n}{2},\; m+\frac{n}{2}+j \\[4pt] m + \frac{n}{2},\; m+\frac{n}{2}+s+1+j \end{matrix} \;;\; 1 \right).
\end{align}
Again, we may apply Thomae's identity \eqref{Thomae}
with parameters:
\begin{align*}
    & a=m+k,\;\; \quad b=k-\tfrac n2,\;\; \quad c=1-\tfrac n2,\\
    & e=m+\tfrac n2+k,\;\; \quad f=s+\tfrac n2+k+1,
\end{align*}
and see that
\begin{align}\label{2.2-2}
    &{}_3F_2 \left( \begin{matrix} m,\; 1-\frac{n}{2},\; m+\frac{n}{2}+j \\[4pt] m+\frac{n}{2},\; m+\frac{n}{2}+s+1+j \end{matrix} \;;\; 1 \right) \nonumber \\
    &\quad = \frac{\Gamma\left(m+\frac{n}{2}\right) \Gamma\left(m+\frac{n}{2}+s+1+j\right) \Gamma(n+s)}{\Gamma\left(m+\frac{n}{2}+j\right) \Gamma(m+n+s) \Gamma\left(s+\frac{n}{2}+1\right)} \nonumber \\
    &\quad\quad \times {}_3F_2 \left( \begin{matrix} n+s,\; -j,\; s+1 \\[4pt] m+n+s,\; s+\frac{n}{2}+1 \end{matrix} \;;\; 1 \right).
\end{align}
Substituting \eqref{2.2-2} into \eqref{2.2-1} we obtain
\begin{align}
    I_m(s) &= \frac{1}{\Gamma\left(\frac{n}{2}\right)} \left( \frac{(n-1)_m}{\left(\frac{n}{2}\right)_m} \right)^2 \frac{\Gamma\left(m + \frac{n}{2}\right) \Gamma(n + s)}{\Gamma(m + n + s)} \nonumber \\
    &\quad \times \sum_{j=0}^{\infty} \frac{(m)_j \left(1 - \frac{n}{2}\right)_j}{\left(m + \frac{n}{2}\right)_j j!} \sum_{l=0}^{j} \frac{(n+s)_l (-j)_l (s+1)_l}{(m+n+s)_l \left(\frac{n}{2} + s + 1\right)_l l!}.
\end{align}
Changing the order of summation, which is justified as the sum over $j$ is finite when $n$ is even, we obtain:
\begin{align} \label{2.2-3}
    I_m(s) &= \frac{1}{\Gamma\left(\frac{n}{2}\right)} \left( \frac{(n-1)_m}{\left(\frac{n}{2}\right)_m} \right)^2 \frac{\Gamma\left(m + \frac{n}{2}\right) \Gamma(n + s)}{\Gamma(m + n + s)} \nonumber \\
    &\quad \times \sum_{l=0}^{\infty} \frac{(n+s)_l (s+1)_l}{(m+n+s)_l \left(\frac{n}{2} + s + 1\right)_l l!} \sum_{j=l}^{\infty} \frac{(m)_j \left(1 - \frac{n}{2}\right)_j}{\left(m + \frac{n}{2}\right)_j j!} (-j)_l.
\end{align}
We introduce $k = j - l$ and subsequently $(-j)_l = (-1)^l \frac{j!}{(j-l)!} = (-1)^l \frac{(l+k)!}{k!}$. 
Using Gauss' summation formula
\begin{equation*}
    {}_2F_1 \left( \begin{matrix} a,\; b \\ c \end{matrix} \;;\; 1 \right) = \frac{\Gamma(c)\Gamma(c - a - b)}{\Gamma(c - a)\Gamma(c - b)}
\end{equation*}
the inner sum in \eqref{2.2-3} can be written as 
\begin{align}\label{pro 2.2 - 4}
    &(-1)^l \frac{(m)_l \left(1 - \frac{n}{2}\right)_l}{\left(m + \frac{n}{2}\right)_l} \sum_{k=0}^{\infty} \frac{(m+l)_k \left(1 - \frac{n}{2} + l\right)_k}{\left(m + \frac{n}{2} + l\right)_k k!} \nonumber \\
    &= (-1)^l \frac{(m)_l \left(1 - \frac{n}{2}\right)_l}{\left(m + \frac{n}{2}\right)_l} \; {}_2F_1 \left( \begin{matrix} m+l,\; 1 - \frac{n}{2} + l \\[4pt] m + \frac{n}{2} + l \end{matrix} \;;\; 1 \right) \nonumber \\
    &= (-1)^l \frac{(m)_l \left(1 - \frac{n}{2}\right)_l}{\left(m + \frac{n}{2}\right)_l} \frac{\Gamma\left(m + \frac{n}{2} + l\right) \Gamma(n - 1 - l)}{\Gamma\left(\frac{n}{2}\right) \Gamma(m + n - 1)}, \nonumber \\
\end{align}
where the term $\Gamma(n-1-l)$ is finite as $l$ runs from $0$ to $1-\frac{n}{2}$.
Using $\Gamma(m + \frac{n}{2} + l) = \Gamma(m+\frac{n}{2})(m + \frac{n}{2})_l$, $\Gamma(n - 1 - l) = \frac{(-1)^l \Gamma(n - 1)}{(2 - n)_l}$, \eqref{pro 2.2 - 4} is equal to
\begin{equation}\label{pro 2.2 - 5}
    \frac{(m)_l \left(1 - \frac{n}{2}\right)_l}{(2-n)_l} \frac{\Gamma\left(m + \frac{n}{2} \right) \Gamma(n - 1)}{\Gamma\left(\frac{n}{2}\right) \Gamma(m + n - 1)}.
\end{equation}
Substituting this into \eqref{2.2-3} and simplifying the Gamma functions and Pochhammer symbols, we obtain
\begin{align}
    I_m(s) &= \frac{(n-1)_m^2}{\left(\frac{n}{2}\right)_m^2}\frac{\Gamma\left(m+\frac{n}{2}\right) \Gamma(n+s)}{\Gamma(m+n+s)} \frac{\Gamma(n-1)\Gamma\left(m+\frac{n}{2}\right)}{\Gamma\left(\frac{n}{2}\right)^2\,\Gamma(m+n-1)}\nonumber \\
    &\quad \quad \times \sum_{l=0}^{\infty} \frac{(m)_l \left(1 - \frac{n}{2}\right)_l (n+s)_l (s+1)_l}{(m + n + s)_l \left(\frac{n}{2} + s + 1\right)_l (2 - n)_l} \frac{1}{l!} \nonumber \\
    &= \frac{(n-1)_m}{(n+s)_m} \; {}_4F_3 \left( \begin{matrix} m,\; 1 - \frac{n}{2},\; n+s,\; s+1 \\[4pt] m + n + s,\; \frac{n}{2} + s + 1,\; 2-n \end{matrix} \;;\; 1 \right),
\end{align}
Thus, our conclusion follows for \(s>-1\).
For \(s<-1\), the result follows by analytic continuation away from the poles of \(I_m(s)\).
\end{proof}

\begin{remark}
    When $n$ is even we may show that Proposition~\ref{Pro 3.1} and Proposition~\ref{Prop 3.2} indeed give the same expression for $I_m(s).$ To do this we invoke the following transformation for balanced terminating $_4F_3$ due to Bailey:
\begin{equation}\label{Bailey}
    {}_4F_3 \left( \begin{matrix} -N, a, b, c \\ d, e, f \end{matrix} ; 1 \right) = \frac{(e-a)_N(f-a)_N}{(e)_N(f)_N} {}_4F_3 \left( \begin{matrix} -N, a, d-b, d-c \\ d, a-e-N+1, a-f-N+1 \end{matrix} ; 1 \right)
\end{equation}
(see~\cite{olver2010nist}, Eq.~(16.4.15), p. 407).
Applying this to \eqref{sude_dva_dva} with parameters
$$N = 1-\frac{n}{2}, \quad a = n+s, \quad b = s+1, \quad c= m,$$
$$d = \frac{n}{2} + s + 1, \quad e = m+n+s, \quad f = 2-n$$
and we get
\begin{align*}
{}_4F_3 \left(
\begin{matrix}
m,\; 1-\frac n2,\; n+s,\; s+1\\[4pt]
m+n+s,\; \frac n2+s+1,\; 2-n
\end{matrix};1
\right)
&=
\frac{
(m)_{\frac n2-1}\,(2-2n-s)_{\frac n2-1}
}{
(m+n+s)_{\frac n2-1}\,(2-n)_{\frac n2-1}
}\\
&\quad \times
{}_4F_3 \left(
\begin{matrix}
1-\frac n2,\; n+s,\; \frac n2+s+1-m,\; \frac n2\\[4pt]
\frac n2+s+1,\; 2-m-\frac n2,\; \frac{3n}{2}+s
\end{matrix};1
\right).
\end{align*}
Upon expanding Pochhammer symbols and multiplying by $\frac{(n-1)_m}{(n+s)_m}$ we see that \eqref{sude_dva_dva} is exactly equal to \eqref{sude}.
\end{remark}

\section{Discrete Wallach sets}
\begin{theorem}
 Suppose that $n \geq6$ is even, then 
 $$
  \lim_{s \to -n+1}(s +n-1)I_m(s)=-\frac{(n-1)_m}{(1)_m}\,\frac{\left(\frac{n}{2}-1\right)_m}{\left(\frac{n}{2}\right)_m}
 $$
 and
 $$
 \lim_{s \to -n+2}(s +n-2)I_m(s) = -\frac{(n-1)_m}{(2)_m}\left(\frac{m(m+1)n\left(\frac{n}{2}-1\right)\left(\frac{n}{2}-2\right)}{2\left(m+\frac{n}{2}-2\right)\left(m+\frac{n}{2}-1\right)\left(m+\frac{n}{2}\right)}\right).
 $$
 In particular, we have $-n+1, \, -n + 2 \in \mathcal{W}^d_1$. For $n=4$, we have $-3 \in \mathcal{W}^d_1.$
\end{theorem}
\begin{proof}
As $n \geq 6$ the function $I_m(s)$ is singular for $s = -n+1$ and $s = -n+2$. From Proposition~\ref{Prop 3.2} we have
\begin{align*}
    \lim_{s \to -n+1}(s +n-1)I_m(s) &=\lim_{s \to -n+1}(s+n-1) \frac{(n-1)_m}{(n+s)_m}\\
    & \qquad\times\;{}_4F_3 \left( \begin{matrix} m,\; 1 - \frac{n}{2},\; n+s,\; s+1 \\[4pt] m + n + s,\; \frac{n}{2} + s + 1,\; 2-n \end{matrix} \;;\; 1 \right)\\
    &= \frac{(n-1)_m}{(1)_m}\\
    & \qquad \times\lim_{s \to -n+1} \sum_{l=0}^{\frac{n}{2}-1} \frac{(m)_l \left(1 - \frac{n}{2}\right)_l (n+s)_l (s+1)_l}{(m + n + s)_l \left(\frac{n}{2} + s + 1\right)_l (2 - n)_l} \frac{1}{l!}.
\end{align*}
The only non-zero term in the sum is the one corresponding to \(l=\frac{n}{2}-1\). Using the fact $$
\lim_{s\to -n+1} \,\frac{(s+n-1)}{\left(\frac n2+s+1\right)_{\frac n2-1}}
=
\frac{(-1)^{\frac n2-2}}{\left(\frac n2-2\right)!}
$$
we obtain
$$
\lim_{s\to -n+1}(s+n-1)I_m(s)
=
-\frac{(n-1)_m}{\Gamma(m)}\,\frac{\left(\frac{n}{2}-1\right)_m}{\left(\frac{n}{2}\right)_m}.
$$
Similarly using 
\[
\lim_{s\to -n+2} \frac{(s+n-2)}{\left(\frac n2+s+1\right)_{\frac n2-1}}
=
\frac{(-1)^{\frac n2-3}}{\left(\frac n2-3\right)!}, \qquad \lim_{s\to -n+2} \frac{(s+n-2)}{\left(\frac n2+s+1\right)_{\frac n2-2}}
=
\frac{(-1)^{\frac n2-3}}{\left(\frac n2-3\right)!}
\]
we obtain 
\begin{align*}
     \lim_{s \to -n+2}(s +n-2)I_m(s) &=\\
    \frac{(n-1)_m}{(2)_m}\lim_{s \to -n+2} &\sum_{l=0}^{\frac{n}{2}-1} \frac{(s+n-2)\,(m)_l \left(1 - \frac{n}{2}\right)_l (n+s)_l (s+1)_l}{(m + n + s)_l \left(\frac{n}{2} + s + 1\right)_l (2 - n)_l} \frac{1}{l!}\\
     &= \frac{(n-1)_m}{(2)_m}\left(- \frac{m(m+1)n\left(\frac{n}{2}-1\right)\left(\frac{n}{2}-2\right)}{4\left(m+\frac{n}{2}-2\right)\left(m+\frac{n}{2}-1\right)} + \frac{m(m+1)n\left(\frac{n}{2}-1\right)\left(\frac{n}{2}-2\right)}{4\left(m+\frac{n}{2}-1\right)\left(m+\frac{n}{2}\right)}\right)\\
    &= -\frac{(n-1)_m}{(2)_m}\left(\frac{m(m+1)n\left(\frac{n}{2}-1\right)\left(\frac{n}{2}-2\right)}{2\left(m+\frac{n}{2}-2\right)\left(m+\frac{n}{2}-1\right)\left(m+\frac{n}{2}\right)}\right).
\end{align*}
Hence, we conclude that
$
-n+1,\,-n+2 \in \mathcal{W}^d_1.
$
When $n =4$, only $-3$ belongs to $\mathcal{W}_1^d$, since $I_m(s)$ is regular at $s = -2$.
\end{proof}

\begin{remark}
    Note that Proposition~\ref{Prop 3.2} shows that for even $n$, $I_m(s)$ has only simple poles at $s = -\frac{n}{2}-1 - k$ for $k \geq 0$, hence $\mathcal{W}^d_k = \emptyset$ for $k > 1$.
\end{remark}

\begin{remark}\label{remark}
    It was shown in \cite[Proposition 1]{BEY2025} that for $m > 0$
    \begin{equation}\label{Dirichlet}
\lim_{s\rightarrow-n}(n+s)\,I_m(s) 
=
\begin{cases}
\dfrac{(n-1)_m}{\Gamma(m)}, & \text{if } n \text{ is even}, \\[2ex]
\dfrac{2(n-1)_m}{\Gamma(m)}, & \text{if } n \text{ is odd},
\end{cases}
\end{equation} 
with the residue at $s = -n-1$ being identical.
We can recover this result from Proposition~\ref{Pro 3.1}. Suppose that $n$ is odd. Then, from Proposition~\ref{Pro 3.1}, we have
\begin{align*} 
I_m(s) &= \frac{(n-1)_m^2}{(n+s)_m}\,
\Gamma\!\left(
\begin{matrix}
s+\frac{n}{2}+1,\; s+2n-1,\; 1-m-\frac{n}{2}
\\
s-m+\frac{n}{2}+1,\; 1-\frac{n}{2},\; m+2n+s-1
\end{matrix}
\right) \\
&\qquad \quad \times {}_4F_3 \left( \begin{matrix} n + s, \, m + n - 1, \, m + s + \frac{3n}{2} - 1, \, m \\ m + n + s, \, m + \frac{n}{2}, \, m + 2n + s - 1 \end{matrix} ; 1 \right) \nonumber \\[1em]
& + (n-1)_m^2  \Gamma\!\left(
\begin{matrix}
\frac{n}{2},\; s+2n-1,\; m+\frac{n}{2}-1,\; s+n
\\
m,\; m+n-1,\; s+\frac{3n}{2},\; m+s+\frac{3n}{2}-1
\end{matrix}
\right) \nonumber \\
&\qquad \quad \times {}_4F_3 \left( \begin{matrix} \frac{n}{2}, \, s - m + \frac{n}{2} + 1, \, 1 - \frac{n}{2}, \, s + n \\ s + \frac{3n}{2}, \, 2 - m - \frac{n}{2}, \, s + \frac{n}{2} + 1 \end{matrix} ; 1 \right) \nonumber.
\end{align*}
Using the facts
$$\lim_{s\rightarrow-n} \frac{(s+n)}{(s+n)_{m+k}} = \frac{1}{(m+k-1)!}, \quad \lim_{s\rightarrow-n} (s+n)\Gamma(s+n) = 1,$$
together with
$$\lim_{s\rightarrow-n-1} \frac{s+n+1}{(s+n)_{m+k}} = -\frac{1}{(m+k-2)!} \quad \text{for $m+k\geq1$}, \quad \lim_{s\rightarrow-n-1} (s+n+1)\Gamma(s+n) = -1$$
one can see, from a direct calculation, that each term in the residue of $I_m(s)$ at  both $s = -n$ and $s = -n-1$ is equal to 
$$
\frac{(n-1)_m}{\Gamma(m)}.
$$
As the first term vanishes when $n$ is even, our conclusion follows.
\end{remark}
\begin{theorem}
    For $n$ odd, we have
    \begin{align*}
  \lim_{s\rightarrow -n+\frac{1}{2}} (s+n-\frac{1}{2})\,I_m(s) &=  - \Gamma\!\left(
\begin{matrix}
m+n-1,\; m-\frac12,\; \frac n2,\; \frac n2,\; n-\frac12,\; n-\frac12
\\[3pt]
\frac12, \; m, \; m, \; \frac{n-1}{2}, \; n, \; n-1, \; n-1, \; \frac{n+1}{2}
\end{matrix}
\right)\\[3pt]
& \quad \times {}_4F_3\left( \begin{matrix} \frac{1}{2}, & \frac{n}{2}, & n-\frac{1}{2}, & 1-m \\ \frac{n+1}{2}, & n, & \frac{3}{2}-m \end{matrix} \;;\; 1 \right),
\end{align*}
and
\begin{align*}
     \lim_{s\rightarrow -n-\frac{1}{2}} (s+n+\frac{1}{2})\,I_m(s) &=
-\Gamma\!\left(
\begin{matrix}
\frac32, \; m+n-1,\; \frac n2,\; n-\frac32,\; n+\frac12,\;
m+\frac n2+\frac12,\; m-\frac32
\\[3pt]
n-1,\; n-1,\; m,\; m,\; m+\frac n2-\frac32,\;
\frac{n+1}{2},\; \frac{n+3}{2},\; 1-\frac n2,\; n
\end{matrix}
\right)
\\[3pt]
&\quad \times
{}_4F_3 \left(
\begin{matrix}
\frac32,\; \frac n2,\; n+\frac12,\; 1-m\\[3pt]
\frac{n+3}{2},\; \frac52-m,\; n
\end{matrix};1
\right).
 \end{align*}
In particular, we have  $-n-\frac{1}{2}, \, -n+\frac{1}{2} \in \mathcal{W}^d_1$.
\end{theorem}
\begin{proof}
        When $n$ is odd, from Proposition~\ref{Pro 3.1} we have
\begin{align} 
I_m(s) &= \frac{(n-1)_m^2}{(n+s)_m}\,
\Gamma\!\left(
\begin{matrix}
s+\frac{n}{2}+1,\; s+2n-1,\; 1-m-\frac{n}{2}
\\
s-m+\frac{n}{2}+1,\; 1-\frac{n}{2},\; m+2n+s-1
\end{matrix}
\right) \label{liche_thm} \\[3pt]
&\qquad \quad \times {}_4F_3 \left( \begin{matrix} n + s, \, m + n - 1, \, m + s + \frac{3n}{2} - 1, \, m \\ m + n + s, \, m + \frac{n}{2}, \, m + 2n + s - 1 \end{matrix} ; 1 \right) \nonumber \\[1em]
& + (n-1)_m^2  \Gamma\!\left(
\begin{matrix}
\frac{n}{2},\; s+2n-1,\; m+\frac{n}{2}-1,\; s+n
\\
m,\; m+n-1,\; s+\frac{3n}{2},\; m+s+\frac{3n}{2}-1
\end{matrix}
\right) \label{sude_thm} \\[3pt]
&\qquad \quad \times {}_4F_3 \left( \begin{matrix} \frac{n}{2}, \, s - m + \frac{n}{2} + 1, \, 1 - \frac{n}{2}, \, s + n \\ s + \frac{3n}{2}, \, 2 - m - \frac{n}{2}, \, s + \frac{n}{2} + 1 \end{matrix} ; 1 \right) \nonumber.
\end{align}
We first compute residue at $s = -n+\frac1 2.$
For $n > 3$ each Gamma function in \eqref{liche_thm} is regular at $s = -n + \frac{1}{2}$. For $n = 3$ we have $$\lim_{s \rightarrow-n+\frac{1}{2}}\frac{\Gamma(s+\frac{n}{2}+1)}{\Gamma(s-m+\frac{n}{2}+1)} = (-1)^m\,m!,$$
hence the first prefactor is regular for each $n \geq 3$.
Also the prefactor of \eqref{sude_thm} is regular, thus 
\begin{align*}
   \lim_{s\rightarrow -n+\frac{1}{2}} (s+n-\frac{1}{2})\,I_m(s) &=  G_m \lim_{s\rightarrow -n+\frac{1}{2}} (s+n-\frac{1}{2})\,{}_4F_3 \left( \begin{matrix} \frac{n}{2}, \, s - m + \frac{n}{2} + 1, \, 1 - \frac{n}{2}, \, s + n \\ s + \frac{3n}{2}, \, 2 - m - \frac{n}{2}, \, s + \frac{n}{2} + 1 \end{matrix} ; 1 \right)
\end{align*}
where 
$$
G_m = (n-1)_m^2\,  \Gamma\!\left(
\begin{matrix}
\frac{n}{2},\; n-\frac{1}{2},\; m+\frac{n}{2}-1,\; \frac{1}{2}
\\
m,\; m+n-1,\; \frac{n+1}{2},\; m+\frac{n}{2}-\frac{1}{2}
\end{matrix}
\right).
$$
Using \eqref{TheIdentity} we get
\begin{align}
    \lim_{s\rightarrow -n+\frac{1}{2}} (s+n-\frac{1}{2})\,&{}_4F_3 \left( \begin{matrix} \frac{n}{2}, \, s - m + \frac{n}{2} + 1, \, 1 - \frac{n}{2}, \, s + n \\ s + \frac{3n}{2}, \, 2 - m - \frac{n}{2}, \, s + \frac{n}{2} + 1 \end{matrix} ; 1 \right)\nonumber\\
    &= 
    -\frac{\Gamma(m + \frac{n-1}{2})
\Gamma\left(n-\frac12\right)
}{\Gamma(m)\,
\Gamma\left(\frac{n-1}{2}\right)
\Gamma\left(1-\frac n2\right)
\Gamma(n)
}\frac{\Gamma\!\left(2-m-\frac{n}{2}\right)}{\Gamma\!\left(\frac{3}{2}-m\right)}\\
& \times
_4F_3\left( \begin{matrix} \frac{1}{2}, & \frac{n}{2}, & n-\frac{1}{2}, & 1-m \\ \frac{n+1}{2}, & n, & \frac{3}{2}-m \end{matrix} \;;\; 1 \right),
\end{align}
here we rewrote Pochhammer symbols as Gamma functions.
Using Legendre's formula \begin{equation}\label{lagrange}
\Gamma(z)\Gamma(1-z)=\frac{\pi}{\sin(\pi z)},
\end{equation} we get 
$$
\frac{\Gamma\!\left(2-m-\frac{n}{2}\right)}{\Gamma\!\left(\frac{3}{2}-m\right)} = \frac{\Gamma\!\left(m - \frac{1}{2}\right)}{\Gamma\!\left(m + \frac{n}{2} - 1\right)}.
$$
Putting all of this together we get
\begin{align*}
  \lim_{s\rightarrow n+\frac{1}{2}} (s+n-\frac{1}{2})\,I_m(s) &=  - \Gamma\!\left(
\begin{matrix}
m+n-1,\; m-\frac12,\; \frac n2,\; \frac n2,\; n-\frac12,\; n-\frac12
\\[3pt]
\frac12, \; m, \; m, \; \frac{n-1}{2}, \; n, \; n-1, \; n-1, \; \frac{n+1}{2}
\end{matrix}
\right)\\
& \quad \times {}_4F_3\left( \begin{matrix} \frac{1}{2}, & \frac{n}{2}, & n-\frac{1}{2}, & 1-m \\ \frac{n+1}{2}, & n, & \frac{3}{2}-m \end{matrix} \;;\; 1 \right),
\end{align*}
 now as $\frac{(1-m)_{\ell}}{(\frac{3}{2}-m)_{\ell}}$ is positive for $0 \leq \ell \leq m-1$ we see that $-n+\frac{1}{2} \in\mathcal{W}^d_1$.
 
Exactly the same argument shows, that 
\begin{align*}
     \lim_{s\rightarrow -n-\frac{1}{2}} (s+n+\frac{1}{2})\,I_m(s) &=
-\Gamma\!\left(
\begin{matrix}
\frac32, \; m+n-1,\; \frac n2,\; n-\frac32,\; n+\frac12,\;
m+\frac n2+\frac12,\; m-\frac32
\\[3pt]
n-1,\; n-1,\; m,\; m,\; m+\frac n2-\frac32,\;
\frac{n+1}{2},\; \frac{n+3}{2},\; 1-\frac n2,\; n
\end{matrix}
\right)
\\[6pt]
&\quad \times
{}_4F_3 \left(
\begin{matrix}
\frac32,\; \frac n2,\; n+\frac12,\; 1-m\\[3pt]
\frac{n+3}{2},\; \frac52-m,\; n
\end{matrix};1
\right),
 \end{align*}
thus $-n-\frac{1}{2} \in \mathcal{W}_1^d$ and our conclusion follows.
\end{proof}
\begin{theorem}
    For $n$ odd, we have
    \begin{align}\nonumber
\lim_{s \rightarrow -2n +1 -k}(s + 2n-1+k)^2\,I_{m}(s) &=\;(-1)^{\frac{n+1}{2}}\,
\frac{}{} \,\frac{2\pi}{k!}\,\frac{(2n+k-2)!}{(n-1)!^2}\,
\frac{\Gamma\!\left(k+\frac n2\right)}{\Gamma\!\left(1-\frac n2\right)^2}\,
\frac{\Gamma(n-1+m)^2}
{\Gamma(n-1)^2\,\Gamma(m)^2\,\Gamma\!\left(\frac{3n}{2}+k-1\right)}\,\\\label{opi}
&\qquad\times
{}_4F_3\!\left(
\begin{matrix}
m+n-1,\; -\frac k2,\; \frac{1-k}{2},\; 1-m \\
n,\; 1-\frac n2-k,\; \frac{n+1}{2}
\end{matrix}
;1
\right).
\end{align}
In particular, $\mathcal{W}^d_2 = \{-2n+1-k, \text{ for $k =0, 1,2, \dots$}\}$.
\end{theorem}

\begin{proof}
    When $n$ is odd, we have 
    \begin{align} 
I_m(s) &= \frac{\Gamma(n-1 + m)^2 \, \Gamma(n+s) \, \Gamma(s+\frac{n}{2}+1) \, \Gamma(s+2n-1) \, \Gamma(1 - m - \frac{n}{2})}{\Gamma(n-1)^2\,\Gamma(m+n+s) \, \Gamma(s - m + \frac{n}{2} + 1) \, \Gamma(1 - \frac{n}{2}) \, \Gamma(m + 2n + s - 1)} \nonumber \\
&\qquad \quad \times {}_4F_3 \left( \begin{matrix} n + s, \, m + n - 1, \, m + s + \frac{3n}{2} - 1, \, m \\ m + n + s, \, m + \frac{n}{2}, \, m + 2n + s - 1 \end{matrix} ; 1 \right) \nonumber \\[1em]
& + \frac{\Gamma(\frac{n}{2}) \, \Gamma(n-1+m)\, \Gamma(s+2n-1) \, \Gamma(m+\frac{n}{2}-1) \, \Gamma(s+n)}{\Gamma(n-1)\,\Gamma(m) \, \Gamma(n-1) \, \Gamma(s+\frac{3n}{2}) \, \Gamma(m+s+\frac{3n}{2}-1)} \nonumber \\
&\qquad \quad \times {}_4F_3 \left( \begin{matrix} \frac{n}{2}, \, s - m + \frac{n}{2} + 1, \, 1 - \frac{n}{2}, \, s + n \\ s + \frac{3n}{2}, \, 2 - m - \frac{n}{2}, \, s + \frac{n}{2} + 1 \end{matrix} ; 1 \right) \nonumber.
\end{align}
Note that as $s$ approaches $-2n +1 -k$ the only singular terms are $\Gamma(s + 2n -1)$ and $\Gamma(n+s)$, with $\lim_{s \to -2n+1 - k} (s + 2n-1 +k)\,\Gamma(s + 2n -1) = \frac{(-1)^k}{k!}$ and $\lim_{s \to -2n+1 - k} (s + 2n-1 +k)\,\Gamma(n+s) = \frac{(-1)^{n+k+1}}{(n+k-1)!}$. 
Hence $ \lim_{s \to -2n+1 - k} (s + 2n-1 +k)^2 I_m(s)$ is continuous with respect to $m$, thus we can introduce small perturbation $m \to m + \epsilon$ and we see that $\lim_{s \to -2n+1-k} (s+2n-1+k)^2 I_{m+\epsilon}(s)$ is equal to
\begin{align}\nonumber
\frac{(-1)^{n+1}}{k!\,(n+k-1)!}&
\Bigg[
\frac{
\Gamma(n-1 + m + \epsilon)^{2}\,
\Gamma\!\left(2-\frac{3n}{2}-k\right)\,
\Gamma\!\left(1-m-\epsilon-\frac{n}{2}\right)
}{\Gamma(n-1)^2
\Gamma\!\left(m+\epsilon-n+1-k\right)\,
\Gamma\!\left(2-\frac{3n}{2}-k-m-\epsilon\right)\,
\Gamma\!\left(1-\frac{n}{2}\right)\,
\Gamma\!\left(m+\epsilon-k\right)
}
\\\label{Liche_7}
&\qquad\qquad
{}_4F_3\!\Biggl(
\begin{matrix}
1-n-k,\; m+\epsilon+n-1,\; m+\epsilon-\frac{n}{2}-k,\; m+\epsilon \\[4 pt]
m+\epsilon-n+1-k,\; m+\epsilon+\frac{n}{2},\; m+\epsilon-k
\end{matrix}
\,;\,1\Biggr)
\\\nonumber
&\qquad\quad+
\frac{
\Gamma\!\left(\frac{n}{2}\right)\,
\Gamma(n-1 +m + \epsilon)\,
\Gamma\!\left(m+\epsilon+\frac{n}{2}-1\right)
}{ \Gamma(n-1) \,
\Gamma(m+\epsilon)\,
\Gamma(n-1)\,
\Gamma\!\left(1-\frac{n}{2}-k\right)\,
\Gamma\!\left(m+\epsilon-\frac{n}{2}-k\right)
}\\\nonumber
&\qquad\qquad
{}_4F_3\!\Biggl(
\begin{matrix}
\frac{n}{2},\; 2-\frac{3n}{2}-k-m-\epsilon,\; 1-\frac{n}{2},\; 1-n-k \\[4 pt]
1-\frac{n}{2}-k,\; 2-m-\epsilon-\frac{n}{2},\; 2-\frac{3n}{2}-k
\end{matrix}
\,;\,1\Biggr)
\Bigg].\\\label{Liche_8}
\end{align}
First we show that the above two terms are identical.
\\
We apply \eqref{Bailey} to the first $_4F_3$ in \eqref{Liche_7} with parameters
\begin{align*}
N &= n +k -1, \quad a = m + \epsilon + n -1, \quad b = m+\epsilon, \quad c = m + \epsilon - \frac{n}{2} - k\\[4 pt]
& d = m + \epsilon, \quad e = m+ \epsilon - n +1 -k, \quad f = m +\epsilon + \frac{n}{2},
\end{align*}
and we see that
\begin{align}\nonumber
&{}_4F_3\!\Biggl(
\begin{matrix}
1-n-k,\; m+\epsilon+n-1,\; m+\epsilon-\frac{n}{2}-k,\; m+\epsilon \\[4 pt]
m+\epsilon-n+1-k,\; m+\epsilon+\frac{n}{2},\; m+\epsilon-k
\end{matrix}
\,;\,1\Biggr) \\\nonumber
&= 
\frac{(2-2n-k)_{n+k-1} \left(1-\frac{n}{2}\right)_{n+k-1}}{(m+\epsilon-n+1-k)_{n+k-1} \left(m+\epsilon+\frac{n}{2}\right)_{n+k-1}} \\\nonumber
&\qquad\qquad \times {}_4F_3\!\Biggl(
\begin{matrix}
1-n-k,\; m+\epsilon+n-1,\; -k,\; \frac{n}{2} \\[4 pt]
m+\epsilon-k,\; n,\; 1-\frac{n}{2}-k
\end{matrix}
\,;\,1\Biggr).\\\label{nevim}
\end{align}
Note that $(2-2n-k)_{n+k-1} = (-1)^{n+k-1}\frac{(2n+k-2)!}{(n-1)!}$ and similarly 

\begin{align*}
\left(1-\frac{n}{2}\right)_{n+k-1} &= \frac{\Gamma(k+\frac{n}{2})}{\Gamma(1-\frac{n}{2})}, \quad (m+\epsilon -n+1 -k)_{n+k-1} = \frac{\Gamma(m+\epsilon)}{\Gamma(m+\epsilon -n+1-k)},\\& \quad \left(m+\epsilon + \frac{n}{2} \right)_{n+k-1} = \frac{\Gamma(m+\epsilon + \frac{3n}{2} +k -1)}{\Gamma(m+\epsilon + \frac{n}{2})}.
\end{align*}
Substituting this back into \eqref{Liche_7} and grouping the Gamma functions in more convenient way, we see that \eqref{Liche_7} is equal to
\begin{align*}
    (-1)^{n+k-1}&\frac{(2n+k-2)!}{(n-1)!}\,
    \frac{\Gamma(k+\frac{n}{2})}{\Gamma(1-\frac{n}{2})^2}\,\frac{\Gamma(n-1 + m + \epsilon)^{2}\,}
    {\Gamma\!\left(m+\epsilon-k\right)\,}\\
    & \times \frac{ \,\Gamma(m+\epsilon + \frac{n}{2})\Gamma\!\left(1-m-\epsilon-\frac{n}{2}\right)}{\Gamma\!\left(2-\frac{3n}{2}-k-m-\epsilon\right)\,\Gamma(m+\epsilon + \frac{3n}{2} +k -1)\,}
    \frac{
\Gamma\!\left(2-\frac{3n}{2}-k\right)\,
}{\Gamma(n-1)^2
\Gamma(m+\epsilon)\,}
\\\nonumber
&\qquad\qquad\times
{}_4F_3\!\Biggl(
\begin{matrix}
1-n-k,\; m+\epsilon+n-1,\; -k,\; \frac{n}{2} \\
m+\epsilon-k,\; n,\; 1-\frac{n}{2}-k
\end{matrix}
\,;\,1\Biggr),
\end{align*}
and from \eqref{lagrange} it follows that fourth fraction is equal to $(-1)^k$.\\

We also apply Bailey's transformation to  \eqref{Liche_8} with parameters
\begin{align*}
    N &= n + k - 1, \quad a = \frac{n}{2}, \quad b = 2 - \frac{3n}{2} - k - m - \epsilon, \quad c = 1 - \frac{n}{2} \\
    &d = 1 - \frac{n}{2} - k, \quad e = 2 - m - \epsilon - \frac{n}{2}, \quad f = 2 - \frac{3n}{2} - k.
\end{align*}
Hence the $_4F_3$ in \eqref{Liche_8} is equal to 
\begin{align*}
&\frac{\left(2-m-\epsilon-n\right)_{n+k-1} \left(2-2n-k\right)_{n+k-1}}{\left(2-m-\epsilon-\frac{n}{2}\right)_{n+k-1} \left(2-\frac{3n}{2}-k\right)_{n+k-1}} 
\,{}_4F_3\Biggl( 
\begin{matrix} 
1-n-k, & \frac{n}{2}, & -k, & n+m+\epsilon-1 \\[4 pt]
1-\frac{n}{2}-k, & m+\epsilon-k, & n 
\end{matrix} 
;\, 1 \Biggr)\\
& \qquad \qquad = (-1)^{n+k-1}\frac{(2n+k-2)!}{(n-1)!}\frac{\Gamma(1-m-\epsilon+k) \, \Gamma\left(2-m-\epsilon-\frac{n}{2}\right) \, \Gamma\left(2-\frac{3n}{2}-k\right)}{\Gamma(2-m-\epsilon-n) \, \Gamma\left(1-m-\epsilon+\frac{n}{2}+k\right) \, \Gamma\left(1-\frac{n}{2}\right)}\\&\qquad \qquad \times {}_4F_3\Biggl( 
\begin{matrix} 
1-n-k, & \frac{n}{2}, & -k, & n+m+\epsilon-1 \\[4 pt]
1-\frac{n}{2}-k, & m+\epsilon-k, & n 
\end{matrix} 
;\, 1 \Biggr).
\end{align*}
Again, substituting this into \eqref{Liche_8} and grouping Gamma functions we obtain 
\begin{align*}
    (-1)^{n+k-1}&\frac{(2n+k-2)!}{(n-1)!} \frac{\Gamma\!\left(\frac n2\right)}
{\Gamma\!\left(1-\frac n2\right)\Gamma\!\left(1-\frac n2-k\right)} \frac{\Gamma(1-m-\epsilon + k)\,\Gamma(n-1+m+\epsilon)}{\Gamma(2-m-\epsilon - n)} \\
    & \times \frac{\Gamma(2-m-\epsilon - \frac{n}{2})\, \Gamma(m+\epsilon + \frac{n}{2} - 1)}{\Gamma(1-m-\epsilon + \frac{n}{2} +k)\, \Gamma(m+\epsilon - \frac{n}{2} -k)}\,\frac{\Gamma(2 - \frac{3n}{2} - k)}{\Gamma(n-1)^2\, \Gamma(m+\epsilon)} \\ &\qquad \qquad \times {}_4F_3\Biggl( 
\begin{matrix} 
1-n-k, & \frac{n}{2}, & -k, & n+m+\epsilon-1 \\ 
1-\frac{n}{2}-k, & m+\epsilon-k, & n 
\end{matrix} 
;\, 1 \Biggr).
\end{align*}
Again, using \eqref{lagrange} we have

$$
\frac{\Gamma\!\left(\frac n2\right)}
{\Gamma\!\left(1-\frac n2\right)\Gamma\!\left(1-\frac n2-k\right)}
=
(-1)^k\,
\frac{\Gamma\!\left(k+\frac n2\right)}
{\Gamma\!\left(1-\frac n2\right)^2},
$$
and similarly
\begin{equation*}
    \frac{\Gamma(1-m-\epsilon + k)}{\Gamma(2-m-\epsilon - n)} \cdot \Gamma(n-1+m+\epsilon) = (-1)^{k}
    \frac{\Gamma(n-1 + m + \epsilon)^{2}\,}
    {\Gamma\!\left(m+\epsilon-k\right)\,},
\end{equation*}
\begin{equation*} 
    \frac{\Gamma(2-m-\epsilon - \frac{n}{2})\, \Gamma(m+\epsilon + \frac{n}{2} - 1)}{\Gamma(1-m-\epsilon + \frac{n}{2} +k)\, \Gamma(m+\epsilon - \frac{n}{2} -k)} = (-1)^k.
\end{equation*}
Thus we see that 
\begin{align}\nonumber
    \lim_{s \rightarrow -2n +1 -k}(s + 2n-1+k)^2\,I_{m + \epsilon}(s) &= \frac{2}{k!\,(n+k-1)!}\,\frac{(2n+k-2)!}{(n-1)!}\,
    \frac{\Gamma(k+\frac{n}{2})}{\Gamma(1-\frac{n}{2})^2}\,\frac{\Gamma(n-1 + m + \epsilon)^{2}\,}
    {\Gamma\!\left(m+\epsilon-k\right)\,}\\
    & \times 
    \frac{
\Gamma\!\left(2-\frac{3n}{2}-k\right)\,
}{\Gamma(n-1)^2
\Gamma(m+\epsilon)\,}
{}_4F_3\!\Biggl(
\begin{matrix}
1-n-k,\; m+\epsilon+n-1,\; -k,\; \frac{n}{2} \\[4 pt]
m+\epsilon-k,\; n,\; 1-\frac{n}{2}-k
\end{matrix}
\,;\,1\Biggr)\label{semi_final}.
\end{align}
Now, we may apply a transformation due to Whipple (see~\cite{erdelyi1}, Eq.~(9), p. 189):
\begin{equation*}
{}_4F_3\!\left(
\begin{matrix}
-N,\, b,\, c,\, d \\[4 pt]
1-N-b,\, 1-N-c,\, w
\end{matrix}
; 1
\right)
=
\frac{(w-d)_N}{(w)_N}\,
{}_5F_4\!\left(
\begin{matrix}
d,\, 1-N-b-c,\, -\tfrac12 N,\, \tfrac12-\tfrac12 N,\, 1-N-w \\[4 pt]
1-N-b,\, 1-N-c,\, \tfrac12(1+d-w-N),\, 1+\tfrac12(d-w-N)
\end{matrix}
; 1
\right)
\end{equation*}
to \eqref{semi_final} with parameters
$$
N=k,
\qquad
b=\frac n2,
\qquad
c=1-n-k,
\qquad
d=m+\epsilon+n-1,
\qquad
w=m+\epsilon-k.
$$
So the $_4F_3$ in \eqref{semi_final} is equal to 
\begin{equation}\label{final_final}
\frac{(1-n-k)_k}{(m+\epsilon-k)_k}
\,{}_4F_3\!\left(
\begin{matrix}
m+\epsilon+n-1,\; -\frac k2,\; \frac{1-k}{2},\; 1-m-\epsilon \\[4 pt]
n,\; 1-\frac n2-k,\; \frac{n+1}{2}
\end{matrix}
;1
\right).
\end{equation}
Note, that $(1-n-k)_k = (-1)^k\,\frac{(n+k-1)!}{(n-1)!}$. Substituting this into \eqref{semi_final}, cancelling Gamma functions, applying \eqref{lagrange} to the term $\Gamma(2-\frac{3n}{2}-k)$ and setting $\epsilon = 0$ we get 
\begin{align}\nonumber
\lim_{s \rightarrow -2n +1 -k}(s + 2n-1+k)^2\,I_{m}(s) &=\;(-1)^{\frac{n+1}{2}}\,
\frac{}{} \,\frac{2\pi}{k!}\,\frac{(2n+k-2)!}{(n-1)!^2}\,
\frac{\Gamma\!\left(k+\frac n2\right)}{\Gamma\!\left(1-\frac n2\right)^2}\,
\frac{\Gamma(n-1+m)^2}
{\Gamma(n-1)^2\,\Gamma(m)^2\,\Gamma\!\left(\frac{3n}{2}+k-1\right)}\,\\\label{opi}
&\qquad\times
{}_4F_3\!\left(
\begin{matrix}
m+n-1,\; -\frac k2,\; \frac{1-k}{2},\; 1-m \\[4 pt]
n,\; 1-\frac n2-k,\; \frac{n+1}{2}
\end{matrix}
;1
\right).
\end{align}
Now, the hypergeometric function $_4F_3$ is equal to
\begin{equation*}
 \sum_{j=0}^{\min\left(\left\lfloor k/2\right\rfloor,\; m-1\right)}
\frac{(m+n-1)_j\left(-\frac k2\right)_j\left(\frac{1-k}{2}\right)_j(1-m)_j}
{(n)_j\left(1-\frac n2-k\right)_j\left(\frac{n+1}{2}\right)_j}\,
\frac{1}{j!}.
\end{equation*}
For $0\leq j \leq \min\left(\left\lfloor k/2\right\rfloor,\; m-1\right)$ we 
have from the Duplication formula $(a)_j\,(a +\frac{1}{2})_j = 4^{-j}\,(2a)_{2j}$
$$
\left(-\frac{k}{2} \right)_j\, \left( \frac{1-k}{2}\right)_j = 4^{-j}(-k)_{2j} = 4^{-j} (-1)^{2j} \frac{k!}{(k-2j)!} > 0,
$$
and similarly for $j$ in this range we have 
$$
\left(1 - \frac{n}{2} -k \right)_j = (-1)^j \left(\frac{n}{2} + k - j\right)_j, \qquad (1-m)_j = (-1)^j\frac{(m-1)!}{(m-1-j)!},
$$
hence the last $_4F_3$ is positive for each $m$. Consequently \eqref{opi} does not change sign and the claim follows.
\end{proof}
\section{Asymptotic expansion of $I_m(s)$}

\begin{proposition} \label{Prop 5.1}
Suppose that \(n\) is odd, \(s+n-1\geq0\), and
$s\neq -\frac n2-k$ for all $k =0,1,2\dots$.\\
Then
\begin{equation}\label{vysledek}
\lim_{m\to\infty} m^{s+1} I_m(s)
=
\Gamma\!\left(
\begin{matrix}
\frac n2,\; s+2n-1,\; 1+\frac{s+n}{2},\; s+\frac n2+1,\; \frac{s+n}{2}
\\[4 pt]
n-1,\; n-1,\; s+n+1,\; n+\frac s2,\; 1+\frac s2
\end{matrix}
\right).
\end{equation}
If \(n\) is even, the same limit holds without the additional assumption \(s+n-1\geq0\).
\end{proposition}

\begin{proof}
Suppose that $n$ is odd and $s+n-1 \geq 0$, then we know, from Proposition~\ref{Pro 3.1} that $I_m(s)$ is equal to
\begin{align} 
I_m(s) &= \frac{\Gamma(n-1 + m)^2 \, \Gamma(n+s) \, \Gamma(s+\frac{n}{2}+1) \, \Gamma(s+2n-1) \, \Gamma(1 - m - \frac{n}{2})}{\Gamma(n-1)^2\,\Gamma(m+n+s) \, \Gamma(s - m + \frac{n}{2} + 1) \, \Gamma(1 - \frac{n}{2}) \, \Gamma(m + 2n + s - 1)} \label{liche_12} \\
&\qquad \quad \times {}_4F_3 \left( \begin{matrix} n + s, \, m + n - 1, \, m + s + \frac{3n}{2} - 1, \, m \\[4 pt] m + n + s, \, m + \frac{n}{2}, \, m + 2n + s - 1 \end{matrix} ; 1 \right) \nonumber \\[1em]
& + \frac{\Gamma(\frac{n}{2}) \, \Gamma(n-1+m) \, \Gamma(s+2n-1) \, \Gamma(m+\frac{n}{2}-1) \, \Gamma(s+n)}{\Gamma(m) \, \Gamma(n-1)^2 \, \Gamma(s+\frac{3n}{2}) \, \Gamma(m+s+\frac{3n}{2}-1)} \label{liche_2} \\
&\qquad \quad \times {}_4F_3 \left( \begin{matrix} \frac{n}{2}, \, s - m + \frac{n}{2} + 1, \, 1 - \frac{n}{2}, \, s + n \\[4 pt] s + \frac{3n}{2}, \, 2 - m - \frac{n}{2}, \, s + \frac{n}{2} + 1 \end{matrix} ; 1 \right) \nonumber.
\end{align}
We first examine the behaviour of the first factor. We may use \eqref{lagrange} to write
\[
\frac{\Gamma\!\left(1-m-\frac n2\right)}
{\Gamma\!\left(1-\frac n2\right)\Gamma\!\left(s-m+\frac n2+1\right)}
=
-\frac{
\Gamma\!\left(\frac n2\right)\Gamma\!\left(m-s-\frac n2\right)
}{
\pi\,\Gamma\!\left(m+\frac n2\right)
}
\,\sin\!\left(\pi\left(s+\frac n2\right)\right),
\]
thus the prefactor of \eqref{liche_12} is equal to 
\begin{align*}
&\sin\!\left(\pi\left(s+\frac n2\right)\right)\frac{\Gamma(n-1 + m)^2 \, \Gamma(n+s) \, \Gamma(s+\frac{n}{2}+1) \, \Gamma(s+2n-1) \,\Gamma\!\left(\frac n2\right)\Gamma\!\left(m-s-\frac n2\right)}{\pi\Gamma(n-1)^2\,\Gamma(m+n+s) \, \Gamma(m + 2n + s - 1)\,\Gamma\!\left(m+\frac n2\right)} \label{liche_1}.
\end{align*}
Now, using \eqref{odhad} the prefactor of \eqref{liche_12} is, for large $m$, equal to 
$$
\sin\!\left(\pi\left(s+\frac n2\right)\right)
\frac{\Gamma(n+s)\,\Gamma\!\left(s+\frac n2+1\right)\,\Gamma(s+2n-1)\,\Gamma\!\left(\frac n2\right)}
{\pi\,\Gamma(n-1)^2}
\, m^{-2n-3s-1} + \,\mathcal{O}(m^{-2n-3s-2}).
$$
To analyse the asymptotic behaviour of the $_4F_3$ in \eqref{liche_12} we invoke the following asymptotic expansion for large parameters (see~\cite{olver2010nist}, Eq.~(16.11.10), p. 412):
\begin{multline}
{}_{{p+1}}F_{p}\!\left(
\begin{matrix}
a_1+r,\ldots,a_{k-1}+r,a_k,\ldots,a_{p+1} \\
b_1+r,\ldots,b_k+r,b_{k+1},\ldots,b_p
\end{matrix}
; z
\right) \\
=
\sum_{i=0}^{\alpha-1}
\frac{
(a_1+r)_i \cdots (a_{k-1}+r)_i (a_k)_i \cdots (a_{p+1})_i
}{
(b_1+r)_i \cdots (b_k+r)_i (b_{k+1})_i \cdots (b_p)_i
}
\frac{z^i}{i!}
+ \mathcal{O}\!\left(\frac{1}{r^\alpha}\right),
\end{multline}
valid for any non-negative integer $\alpha$ and $k$ from 1 to $p$. Applying this to \eqref{liche_12}, with $p = 3$ and $\alpha = 1$ we see that the $_4F_3$ is equal to $\mathcal{O}(1)$.
Thus \eqref{liche_12} is equal to $\mathcal{O}(m^{-2n-3s - 1})$.\\
Again, using estimate \eqref{odhad} we see that the Gamma prefactor of \eqref{liche_2} is approximately equal to 
$$
\frac{\Gamma\!\left(\frac{n}{2}\right)\Gamma(s+2n-1)\Gamma(s+n)}
{\Gamma(n-1)^2\Gamma\!\left(s+\frac{3n}{2}\right)}
\, m^{-s-1}\biggl(1+\mathcal{O}\left(\frac{1}{m}\right)\biggr).
$$
Now, the second hypergeometric function is equal to 
\begin{equation} \label{druhy_clen}
\begin{aligned}
{}_4F_3 \left(
\begin{matrix}
\frac{n}{2}, \, s - m + \frac{n}{2} + 1, \, 1 - \frac{n}{2}, \, s + n \\
s + \frac{3n}{2}, \, 2 - m - \frac{n}{2}, \, s + \frac{n}{2} + 1
\end{matrix}
; 1 \right)
=
\sum_{k=0}^{\infty}
\frac{
\left(\frac{n}{2}\right)_k
\left(s - m + \frac{n}{2} + 1\right)_k
\left(1 - \frac{n}{2}\right)_k
(s+n)_k
}{
\left(s + \frac{3n}{2}\right)_k
\left(2 - m - \frac{n}{2}\right)_k
\left(s + \frac{n}{2} + 1\right)_k
}
\frac{1}{k!}.
\end{aligned}
\end{equation}
We now provide an upper bound, independent of $m$, for \eqref{druhy_clen}.\\ Notice that
\begin{align*}
\left|
\frac{
\left(s-m+\frac n2+1\right)_k
}{
\left(2-m-\frac n2\right)_k
}
\right|
&=
\prod_{j=0}^{k-1}
\left|
1+
\frac{s+n-1}{j+2-m-\frac n2}
\right|\\
&=
\prod_{\substack{0\le j\le k-1\\ j+2-m-\frac n2>0}}
\Bigg|
1+
\frac{s+n-1}{j+2-m-\frac n2}
\Bigg|\,
\prod_{\substack{0\le j\le k-1\\ j+2-m-\frac n2<0}}
\Bigg|
1+
\frac{s+n-1}{j+2-m-\frac n2}
\Bigg|\\
&\leq
\prod_{\ell=0}^{k-1}
\Bigg|
1+
\frac{s+n-1}{\ell+\frac{1}{2}}
\Bigg|\,
\prod_{\substack{0\le j\le k-1\\ j+2-m-\frac n2<0}}
\Bigg|
1+
\frac{s+n-1}{j+2-m-\frac n2}
\Bigg|.
\end{align*}
Using the estimate $\log(1+x) < x$ for $x>0$, we get
\begin{align*}
\log
\prod_{\ell=0}^{k-1}
\left(
1+
\frac{s+n-1}{\ell+\frac12}
\right)
&=
\sum_{\ell=0}^{k-1}
\log
\left(
1+
\frac{s+n-1}{\ell+\frac12}
\right)\\
&\le
(s+n-1)
\sum_{\ell=0}^{k-1}
\frac{1}{\ell+\frac12}\\
&\leq
(s+n-1)\log(1+k)+C,
\end{align*}
for $s+n-1\geq0$. Thus, the first product can be bounded by
$$
C(s,n)(1+k)^{s+n-1}.
$$
Now
$$
\prod_{\substack{0\le j\le k-1\\ j+2-m-\frac n2<0}}
\Bigg|
1+
\frac{s+n-1}{j+2-m-\frac n2}
\Bigg|
=
\prod_{\substack{0\le j\le k-1\\ j+2-m-\frac n2<0}}
\Bigg|
1-
\frac{s+n-1}{\left|j+2-m-\frac n2\right|}
\Bigg| < C'(n,s).
$$
Indeed, whenever 
$
s+n-1 < \left|j+2-m-\frac n2\right|,
$
the corresponding factor is bounded by $1$. Hence, it remains only to consider those indices $j$ for which
$
\left|j+2-m-\frac n2\right| \leq s+n-1.
$
There are only finitely many such indices, at most
$
2\lceil s+n-1\rceil+1
$.

Hence, \eqref{druhy_clen} can be uniformly bounded from above by
\[
C
\sum_{k=0}^{\infty}
\Bigg|
\frac{
\left(\frac{n}{2}\right)_k
\left(1 - \frac{n}{2}\right)_k
(s+n)_k
}{
\left(s + \frac{3n}{2}\right)_k
\left(s + \frac{n}{2} + 1\right)_k
}
\frac{(1+k)^{s+n-1}}{k!}
\Bigg|.
\]
Using \eqref{odhad}, it follows that this sum is absolutely convergent.\\
Thus, from Tannery's theorem \cite[p.~136, \S 49]{BRO} we have
\begin{align*}
    \lim_{m\rightarrow\infty} &{}_4F_3 \left(
\begin{matrix}
\frac{n}{2}, \, s - m + \frac{n}{2} + 1, \, 1 - \frac{n}{2}, \, s + n \\
s + \frac{3n}{2}, \, 2 - m - \frac{n}{2}, \, s + \frac{n}{2} + 1
\end{matrix}
; 1 \right)\\
&= \sum_{k=0}^{\infty} \lim_{m\rightarrow\infty}\frac{\left(s - m + \frac{n}{2} + 1\right)_k}{\left(2 - m - \frac{n}{2}\right)_k}
\frac{
\left(\frac{n}{2}\right)_k
\left(1 - \frac{n}{2}\right)_k
(s+n)_k
}{
\left(s + \frac{3n}{2}\right)_k
\left(s + \frac{n}{2} + 1\right)_k
}
\frac{1}{k!}\\
& = {}_3F_2\!\left(
\begin{matrix}
\frac n2,\; 1-\frac n2,\; s+n, \\[4pt]\label{564}
s+\frac{3n}{2},\; s+\frac n2+1\;
\end{matrix};1
\right),
\end{align*}
as 
$$\lim_{m\rightarrow\infty}\frac{\left(s - m + \frac{n}{2} + 1\right)_k}{\left(2 - m - \frac{n}{2}\right)_k}=1.$$
Therefore 
\begin{equation}\label{final_liche}
  \lim_{m \rightarrow \infty} m^{s+1}I_m(s) =  \frac{\Gamma\!\left(\frac{n}{2}\right)\Gamma(s+2n-1)\Gamma(s+n)}
{\Gamma(n-1)^2\Gamma\!\left(s+\frac{3n}{2}\right)} {}_3F_2\!\left(
\begin{matrix}
\frac n2,\; 1-\frac n2,\; s+n, \\
s+\frac{3n}{2},\; s+\frac n2+1\;
\end{matrix};1
\right).
\end{equation}
Using Dixon's formula (see~\cite{erdelyi1}, Eq.~(6), p. 189)
\begin{equation*}
{}_3F_2\!\left(
\begin{matrix}
a,\ b,\ c \\
1+a-b,\ 1+a-c
\end{matrix}
;1\right)
=
\frac{
\Gamma\!\left(1+\frac{a}{2}\right)
\Gamma(1+a-b)
\Gamma(1+a-c)
\Gamma\!\left(1-b-c+\frac{a}{2}\right)
}{
\Gamma(1+a)
\Gamma\!\left(1-b+\frac{a}{2}\right)
\Gamma\!\left(1-c+\frac{a}{2}\right)
\Gamma(1+a-b-c)
},
\end{equation*}
with $a = s + n$, $b = 1 - \frac n2$ and $c = \frac n2$, we get
\begin{align}\nonumber
    \lim_{m \rightarrow \infty} m^{s+1}I_m(s) &= \frac{\Gamma\!\left(\frac{n}{2}\right)\Gamma(s+2n-1)\Gamma(s+n)}
{\Gamma(n-1)^2\Gamma\!\left(s+\frac{3n}{2}\right)} {}_3F_2\!\left(
\begin{matrix}
\frac n2,\; 1-\frac n2,\; s+n, \\[4pt]
s+\frac{3n}{2},\; s+\frac n2+1,\;
\end{matrix};1
\right) \\\nonumber
&= \frac{\Gamma\!\left(\frac{n}{2}\right)\Gamma(s+2n-1)\Gamma(s+n)}
{\Gamma(n-1)^2\Gamma\!\left(s+\frac{3n}{2}\right)} \left(\frac{
\Gamma\!\left(1+\frac{s+n}{2}\right)
\Gamma\!\left(s+\frac{3n}{2}\right)
\Gamma\!\left(s+\frac n2+1\right)
\Gamma\!\left(\frac{s+n}{2}\right)
}{
\Gamma(s+n+1)\,
\Gamma\!\left(n+\frac s2\right)\,
\Gamma\!\left(1+\frac s2\right)\,
\Gamma(s+n)
}\right). \\\label{Final_limit}
\end{align}
When $n$ is even, the same argument shows that $m^{s+1}I_m(s)$ is equal to \eqref{Final_limit} without the restriction that $s +n-1\geq0$ as the term \eqref{liche_12} is missing and \eqref{liche_2} is uniformly convergent, thus the claim follows.
\end{proof}
We remind the reader that the spaces \(\mathcal{H}_{\#s}\) were defined in the introduction; see~\eqref{definiceH}. Moreover, \(\mathcal{H}_s\) denotes the Bergman spaces of \(H\)-harmonic functions defined in~\eqref{Bergman}.
\begin{corollary}\label{final_vysledek}
Suppose that \(n\) is odd. If \(s+n-1\geq0\) and \(s\neq -2-k\), $s\neq -\frac n2-k$ for $k \geq 0$ then
\(\mathcal{H}_{\#s}\) and \(\mathcal{H}_s\) coincide as spaces, with equivalent norms.
If \(n\) is even, the same conclusion holds without the additional assumption
$s +n-1\geq0$.
\end{corollary}

\begin{proof}
Direct consequence of Proposition~\ref{Prop 5.1}.
\end{proof}
\section{Acknowledgement}
Here I would like to thank my supervisor Miroslav Engliš for introducing me to this topic and for his many valuable remarks. \\Research was supported by GA CR grant no. 25-18042S.


\begin{thebibliography}{9}
\bibitem{erdelyi1}
  Erdélyi, Arthur, et al. 
  \textit{Higher Transcendental Functions}. Vol. 1. 
  McGraw-Hill, 1953.

\bibitem{bailey1935}
W. N. Bailey,
\textit{Generalized Hypergeometric Series},
Cambridge Tracts in Mathematics and Mathematical Physics, No. 32,
Cambridge University Press, Cambridge, 1935.

\bibitem{andrews1999}
  G.~E. Andrews, R.~Askey, and R.~Roy,
  \textit{Special Functions},
  Encyclopedia of Mathematics and its Applications, Vol. 71,
  Cambridge University Press, Cambridge, 1999.
  
  \bibitem{pitre1996}
  S.~N. Pitre and J. Van der Jeugt,
  \textit{Transformation and summation formulas for Kamp{\'e} de F{\'e}riet series},
  Journal of Mathematical Analysis and Applications,
  Vol. 199, No. 1, pp. 276--285, 1996.

  
\bibitem{olver2010nist}
F. W. J. Olver, D. W. Lozier, R. F. Boisvert, and C. W. Clark, eds.,
\textit{NIST Handbook of Mathematical Functions}.
New York: Cambridge University Press, 2010.

\bibitem{Axler}
S. Axler, P. Bourdon, and W. Ramey,
\emph{Harmonic Function Theory},
2nd ed., Springer, 2001.

\bibitem{Stoll2016}
M. Stoll,
\textit{Harmonic and Subharmonic Function Theory on the Hyperbolic Ball}.
Cambridge University Press, 2016.
doi: 10.1017/CBO9781316341063.

\bibitem{BEY2025}
P. Blaschke, M. Engli\v{s}, and E.-H. Youssfi,
A Moebius invariant space of H-harmonic functions on the ball,
\emph{Journal of Functional Analysis} \textbf{288} (2025), no.~9, 110857.
doi: \href{https://doi.org/10.1016/j.jfa.2025.110857}{10.1016/j.jfa.2025.110857}.

\bibitem{Ureyen2023}
A.~E.~Ureyen,
``H-harmonic Bergman Projection on the Real Hyperbolic Ball,''
\emph{Journal of Mathematical Analysis and Applications},
Vol.~519, 2023.

\bibitem{Stoll2019}
M.~Stoll,
``The Reproducing Kernel and Radial Eigenfunctions of the Hyperbolic Laplacian,''
\emph{Mathematica Scandinavica},
Vol.~124, No.~1, pp.~81--101, 2019.

\bibitem{EYM}
M. Engliš and E.-H. Youssfi,
\emph{The M-harmonic Dirichlet space on the ball},
Journal of Mathematical Analysis and Applications,
vol. 545, no. 2, article 129165, 2025.
doi: \url{https://doi.org/10.1016/j.jmaa.2024.129165}.

\bibitem{Aro}
N. Aronszajn,
\emph{Theory of reproducing kernels},
Transactions of the American Mathematical Society,
vol. 68, no. 3, pp. 337--404, 1950.
doi: \url{https://doi.org/10.1090/S0002-9947-1950-0051437-7}.

\bibitem{Rossi}
M.~Vergne and H.~Rossi,
``Analytic continuation of the holomorphic discrete series of a semi-simple Lie group,''
\emph{Acta Mathematica}, vol.~136, pp.~1--59, 1976.
doi:10.1007/BF02392016.


\bibitem{BRO}
Bromwich, T.~J.~I'A. \textit{An Introduction to the Theory of Infinite Series}. 2nd ed., London: Macmillan and Co., 1926.

\end{thebibliography}
\end{document}